\documentclass[preprint,12pt,authoryear
]{elsarticle}

\setcitestyle{numbers,square,sort&compress,comma}

\usepackage{amssymb}
\usepackage{amsmath}
\usepackage{graphicx}
\usepackage{float}
\usepackage{subfig}
\usepackage{amsthm}
\usepackage{color}

\journal{Neurocomputing}

\begin{document}
\newtheorem{theorem}{Theorem}
\newtheorem{lemma}{Lemma}
\renewcommand{\proofname}{\textbf{Proof}}
\begin{frontmatter}



\title{A Generalized Spiking Locally Competitive Algorithm for Multiple Optimization Problems}


\author[label1]{Xuexing Du}
\author[label1]{Zhong-qi K. Tian}
\author[label1]{Songting Li\corref{cor1}}
\author[label1]{Douglas Zhou\corref{cor1}}

\affiliation[label1]{organization={School of Mathematical Sciences, MOE-LSC, and Institute of Natural Sciences, Shanghai Jiao Tong University},
            addressline={800 Dongchuan Rd},
            city={Shanghai},
            postcode={200240},
            state={Shanghai},
            country={China}}

\cortext[cor1]{To whom correspondence may be addressed. Email: songting@sjtu.edu.cn, zdz@sjtu.edu.cn.}




\begin{abstract}
We introduce a generalized Spiking Locally Competitive Algorithm (LCA) that is biologically plausible and exhibits adaptability to a large variety of neuron models and network connectivity structures. In addition, we provide theoretical evidence demonstrating the algorithm's convergence in optimization problems of signal recovery. Furthermore, our algorithm demonstrates superior performance over traditional optimization methods, such as FISTA, particularly by achieving faster early convergence in practical scenarios including signal denoising, seismic wave detection, and computed tomography reconstruction. Notably, our algorithm is compatible with neuromorphic chips, such as Loihi, facilitating efficient multitasking within the same chip architecture—a capability not present in existing algorithms. These advancements make our generalized Spiking LCA a promising solution for real-world applications, offering significant improvements in execution speed and flexibility for neuromorphic computing systems.

\end{abstract}





\end{frontmatter}


\section{Introduction}
In diverse areas such as compressed sensing, Bayesian inference, and dictionary learning, the pursuit of sparse representation of signals is critical for enhancing information transfer, reducing complexity, and optimizing resource use \cite{de2019compressed, zhang2023physics, chai2022image, kuzin2019bayesian, aitchison2017or, mairal2009online}. These disciplines are often confronted with challenging optimization problems, propelling the advancement of efficient solutions. Traditional methods like gradient descent and greedy algorithms have been effective across a variety of optimization challenges \cite{beck2009fast, xiang2021fista, yi2022improved}. However, their efficiency diminishes in the context of large-scale problems, leading to significant computational demands and resource consumption. This limitation has prompted researchers to explore alternative strategies that can more effectively manage large-scale issues. 

An intriguing approach draws inspiration from the human brain, recognized for its exceptional energy efficiency and adaptability. Studies on the primary visual cortex indicate that sensory neurons can encode natural stimuli, such as visual images, with impressive efficiency through sparse coding \cite{field1994goal, olshausen1996emergence, rehn2007network,sachdev2012surround,zhu2013visual, barranca2019compressive, boutin2021sparse}. This has led to the development of neural network models aimed at solving optimization problems in a more energy-efficient manner \cite{barranca2014sparsity, shapero2014optimal, tang2017sparse, barranca2019role, barranca2021neural}. The Spiking Locally Competitive Algorithm (Spiking LCA), a prominent algorithm in unsupervised learning, stands out in this regard \cite{shapero2014optimal}. Yet, Spiking LCA's effectiveness is constrained by its rigidity. Firstly, its reliance on fully inhibitory connections between neurons restricts the scope of optimization problems it can effectively tackle, e.g., the measurement matrix is subject to stringent restrictions. Secondly, distinct neuron network architectures are required for different optimization problems, implying that when we implement the Spiking LCA on neuromorphic chips, network architecture modifications become essential when dealing with varying problems. Therefore, there is an urgent need for an algorithm that can seamlessly adapt to various optimization challenges within a single network framework.

This paper addresses the abovementioned challenges by presenting a new algorithm designed for constructing spiking neural networks, which supports both excitatory and inhibitory neuronal connections. Notably, our approach facilitates the handling of diverse optimization problems within a single network framework by modulating the external input currents to neurons. This adaptability offers significant engineering advantages, particularly the capability to execute multiple tasks within a single chip architecture. Our spiking neural network model is grounded in biologically plausible neuron models, extending from the simple Leaky Integrate-and-Fire (LIF) model to more complex Hodgkin-Huxley type models. We also provide theoretical evidence that our network's firing rates converge to optimal solutions for a variety of  optimization problems, such as LASSO and Elastic-Net \cite{zhang2020lasso, zou2005regularization, he2022granular}. A notable strength of our proposed model is its considerably faster early convergence when compared to leading optimization algorithms like FISTA \cite{beck2009fast}. This enhanced convergence speed enables our model to reach reliable solutions more swiftly, aligning well with practical scenarios where energy efficiency and time are of the essence \cite{davies2018loihi, davies2021advancing, fair2019sparse, watkins2020using, wu2023forward}.

The rest of this paper is structured as follows: Section 2 presents the optimization problems central to our research. Section 3 provides an overview of the generalized Spiking LCA, along with theoretical demonstrations of its convergence across various optimization scenarios. Section 4 compares our algorithm with a traditional widely used optimization algorithm FISTA \cite{beck2009fast} in solving practical problems within compressed sensing and signal processing domains. Lastly, we broaden the application of our algorithm to encompass additional types of real-world optimization problems.

\subsection{Related Works}

The Spiking LCA was primarily developed to tackle the constrained LASSO optimization problem \cite{shapero2014optimal}. The algorithm's ability to converge to the precise solution of the constrained LASSO problem was theoretically proven in Ref.~\cite{tang2017sparse}. Moreover, a rigorous analysis of the convergence rate, which enhanced our understanding of the computational capabilities of SNNs, was provided in Ref.~\cite{chou2018algorithmic}. Subsequent research has effectively extended the Spiking LCA on neuromorphic hardware to address practical problems \cite{davies2018loihi, davies2021advancing, fair2019sparse, watkins2020using, henke2022apples, wu2023forward}. However, in these studies, the hardware was limited to solving a single type of optimization problem at a time, as addressing different optimization problems required alterations to the chip architecture \cite{chavez2023spiking, zins2023neuromorphic, chavez2022neuromorphic}. Additionally, the neuron model employed in the Spiking LCA was based on a capacitor circuit, which contrasts with the resistor-capacitor (RC) circuit models that more accurately represent real neurons.
Our work addresses these issues by developing a generalized Spiking LCA. This algorithm is  designed to support efficient multitasking within the same chip architecture, while ensuring compatibility with diverse neuron models and diverse connectivity patterns among neurons.

\section{Sparse Approximation and Recovery Problems}

Sparse approximation and recovery problems are fundamental problems in signal processing and machine learning, aiming to exploit signals' sparsity for various applications. These problems have attracted significant attention in recent years due to their wide range of applications, including image processing, seismic wave detection, and feature selection \cite{kougioumtzoglou2020sparse, wang2020structured, wang2022interpolation, xu2022sparse}.

Sparse approximation primarily represents a given signal using a few non-zero coefficients from an overcomplete dictionary. It seeks the optimal sparse linear combination of atoms (basis functions) in the dictionary that approximates the given signal. In sparse approximation problems, we assume that the signal comprises structured components and unstructured additive noise, as expressed in Eq. \ref{linear_genrative},

\begin{equation}
    \label{linear_genrative}
    \mathbf{s}
 = \Phi \mathbf{a} + \mathbf{\epsilon},
\end{equation}
where a vector input $ \mathbf{s} \in \mathbb{R}^M$ corresponds to an input signal from a particular class of signals. It is a linear combination of an overcomplete dictionary $\Phi = [\phi_1, \phi_2, \cdots, \phi_N]$ (a dictionary with more atoms than the input signal dimension) using coefficients $\mathbf{a} \in \mathbb{R}^N$. Furthermore, $
\mathbf{\epsilon}$ represents additive Gaussian white noise.

While the sparse approximation problem focuses on finding the optimal sparse representation $\mathbf{a}$ for a given signal $\mathbf{s}$ using an overcomplete dictionary, sparse recovery aims to reconstruct a sparse signal from a set of limited, noisy, and underdetermined measurements. This problem frequently arises in compressed sensing, which aims to accurately recover the original sparse signal from a smaller number of linear measurements than those required by the Nyquist-Shannon sampling theorem. The primary concept behind compressed sensing is to exploit the inherent sparsity or compressibility in a specific transformation domain, such as the wavelet or Fourier domain \cite{donoho2006compressed, li2021double, ueda2021compressed}. Mathematically, compressed sensing is described as follows. Consider a signal $\mathbf{x} \in \mathbb{R}^N$, which is $K$-sparse in a transformation domain $\Psi$, i.e., only $K$ coefficients in the transformed domain are non-zero ($K \ll N$). This can be expressed as $\mathbf{x} = \Psi \mathbf{a}$, where $\mathbf{a} \in \mathbb{R}^N$ is the $K$-sparse coefficient vector. The signal $\mathbf{x}$ can be measured using an $M \times N$ measurement matrix $\Phi$, with $M \ll N$
, resulting in a compressed measurement vector $\mathbf{s} \in \mathbb{R}^M$, satisfying: $\mathbf{s} = \Phi \mathbf{x} = \Phi \Psi \mathbf{a}$. Since $\Phi$ and $\Psi$ are incoherent, i.e., their columns are not correlated, the product of the two matrices $A = \Phi \Psi$ can be treated as a new sensing matrix. The goal is to recover the sparse coefficient vector $\mathbf{a}$ from the compressed measurement vector $\mathbf{s}$.

The algorithms used to solve sparse recovery and sparse approximation problems are often the same. In this paper, $A$ represents either $\Phi$ or $\Phi \Psi$ collectively. We aim to solve the underdetermined system $\mathbf{s} = A \mathbf{a} + \epsilon$, with the prior knowledge that only a few entries in $\mathbf{a}$ are non-zero. This problem is mathematically formulated as follows:

\begin{equation}
\label{origianl_function}
\min _{\mathbf{a}} E=\frac{1}{2}\|\mathbf{s}-A \mathbf{a}\|_2^2+\lambda \widetilde{C}(\mathbf{a}),
\end{equation}
where the objective function in Eq. \ref{origianl_function} comprises two terms. The first term measures the mean squared reconstruction error (MSR) of signals while the second term $\widetilde{C}(\mathbf{a})$ imposes a penalty that promotes sparsity in the coefficient vector. The parameter $\lambda$ balances the trade-off between data fidelity and sparsity. However, directly optimizing this objective function with the $\ell^{0}$-norm, $\widetilde{C}(\mathbf{a})=\|\mathbf{a}\|_0$, as the sparsity-inducing function is computationally intractable, i.e., an NP-hard problem. Therefore, alternative surrogate sparsity-inducing functions are commonly used, with the $\ell^{1}$-norm, $\widetilde{C}(\mathbf{a})=\|\mathbf{a}\|_1$, a convex function that encourages sparsity, being a popular solution. The resulting optimization problem, known as LASSO \cite{zhang2020lasso}, is expressed as:

\begin{equation}
\label{LASSO}
\min _{\mathbf{a}} E=\frac{1}{2}\|\mathbf{s}-A \mathbf{a}\|_2^2+\lambda  \|\mathbf{a}\|_1.
\end{equation}
In some problems, we have additional requirement on the variable $\mathbf{a}$ to be non-negative, i.e., $\mathbf{a} \geq 0$. This variation is referred to as the constrained LASSO problem:
\begin{equation}
\label{CLASSO}
\min _{\mathbf{a} \ge 0} E=\frac{1}{2}\|\mathbf{s}-A \mathbf{a}\|_2^2+\lambda  \|\mathbf{a}\|_1.
\end{equation}


Although there have been significant advancements in solving the LASSO problem, such as ISTA, FISTA, and LISTA \cite{beck2009fast, gregor2010learning}, its computational complexity is a major challenge for real-time digital signal processing applications that deal with large-scale signals. Thus, the computational demands of the LASSO problem limit its practical applicability, where rapid and low-power reconstruction algorithms are crucial. Therefore, there is a growing demand for efficient algorithms and hardware architectures that enable real-time processing of large-scale signals or data. Indeed, developing such technologies is vital for advancing signal processing and machine learning, and with real-world  applications.


\section{The generalized Spiking LCA}

\subsection{Review of LCA and Spiking LCA}

To elucidate our approach, we first provide an overview of the Locally Competitive Algorithm (LCA) and its spiking variant, the Spiking LCA. The LCA, often termed Analog LCA, is known for its robust convergence characteristics and capability to tackle large-scale challenges \cite{rozell2008sparse}. The LCA's architecture comprises an interconnected neural network encompassing $N$ neurons. The LCA can solve the LASSO problem represented in Eq. \ref{LASSO} by evolving the dynamics of the neuron's membrane potential described by  
\begin{equation}
\label{lca}
\begin{aligned}
&\dot{\mathbf{u}}(t) =\frac{1}{\tau} [\mathbf{b}-\mathbf{u}(t)-\left(A^{T} A-I\right) \mathbf{a}(t)], \\
&\mathbf{a}(t) =T_{\pm \lambda}(\mathbf{u}(t)), \quad T_{\pm \lambda}(\mathbf{u}) = T_{\lambda}(\mathbf{u}) + T_{\lambda}(-\mathbf{u})\\
&T_{\lambda}(\mathbf{u}(t)) = \max (|\mathbf{u}|-\lambda, 0) \cdot \operatorname{sgn}(\mathbf{u}).
\end{aligned}
\end{equation}
In the above neuronal dynamics model, each neuron receives a constant input $\mathbf{b} \in \mathbb{R}^{M
}$, which is detemined as $\mathbf{b}=A^{T} \mathbf{s}$. Here, the matrix $A$ and vector $\mathbf{s}$ are as defined in the original Lasso problem (Eq.~\ref{LASSO}). When the firing threshold of the neuron is reached, the neuron dispatches inhibitory signals to its counterparts. $W = I - A^{T} A$ characterizes this inter-neuronal interaction strength, and $\tau$ denotes the time constant of neuronal response. The neurons communicate through activations $\mathbf{a}(t)$, akin to spike rates. The function $T_{\lambda}(\cdot)$  enforces output sparsity via a soft-thresholding mechanism. Recent empirical evidence suggests that LCA demonstrates local asymptotic stability, ensuring resilience against external disturbances and convergence to equilibrium states over time \cite{balavoine2012convergence}. Consequently, for specific inputs, the system converges to a unique and stable solution consistent with the global minimum of a LASSO optimization problem \cite{balavoine2012convergence}.

Despite the efficacy of LCA in solving large-scale challenges, it has certain limitations. For example, its dependence on continuous-time dynamics may enhance computational and energy costs, especially on traditional computing platforms. Therefore, the Spiking LCA is conceived to circumvent these challenges and harness the potential of neuromorphic hardware \cite{davies2018loihi,zhang2022spiking}. The Spiking LCA harnesses the energy-efficient properties inherent to spiking neural encoding by integrating spike-driven neuronal dynamics into the LCA paradigm. This integration significantly enhances power efficiency by capitalizing on the strengths of spiking neural networks (SNNs)~\cite{tang2017sparse,davies2018loihi,fair2019sparse,zhang2022spiking}. Note that the Spiking LCA addresses the constrained LASSO specifically because of the inherently non-negative firing rates of neurons. 

To implement the Spiking LCA in an SNN, each of the $N$ neurons receives an somatic input current $\mu_i(t)$ over time $t$ to change its membrane potential $v_i(t)$.  The membrane potential accumulates according to the equation \begin{equation}
\label{eqn:slca_v}
v_i(t) = \int_0^t (\mu_i - \lambda)dt,
\end{equation}
while it remains below the firing threshold $v_{th}$, and $\lambda \geq 0$ represents a predefined bias current. This bias current is set as the constant $\lambda$ specified in Eq.~\ref{CLASSO}. When $v_i(t)$ reaches $v_{th}$ at  time $t_{i, sp}$, neuron $i$ is said to fire a spike, and $v_i(t)$ is set to the value of the reset voltage $v_{reset}$. At the same time, inhibitory currents are injected into all other neurons connected with neuron $i$. In the  numerical simulation, the non-dimensional values $v_{reset}=0$,$v_{th}=1$ are used. 

The somatic input current $\mu_i(t)$ consists of a constant background input current $b_i = A^T_i\mathbf{s}$ and synaptic input currents from other neurons, which is governed by
\begin{equation}
\label{mu_diff}
\dot{\mu}_{i}(t) =b_{i}-\mu_{i}(t)-\sum_{j \neq i} w_{i j} \sigma_{j}(t).
\end{equation}
where $w_{ij} = A_i^TA_j$ is the synaptic weight from neuron $j$ to neuron $i$, $\sigma_j(t) = \sum_k \delta(t - t_{j, k})$ is the sum of Dirac delta
functions, and $t_{j,k}$ corresponds to the $k^{th}$ spike time of the $j^{th}$ neuron. And $\alpha(t) = e^{-t}$ for $t \geq 0$ and zero otherwise, implying that the synaptic current is modulated by a weighted exponential decay function when an input is received, consistent with experimental observation. Eqs. \ref{eqn:slca_v}-\ref{mu_diff}, along with the definition of the spike trains $\sigma_i(t)$, describe the Spiking LCA. 

To demonstrate the algorithm's convergence, we introduce two variables, the spike rate $a_i(t)$ and the average somatic input current $u_i(t)$, which are defined below:

\begin{equation}
\label{au_def}
\begin{aligned}
a_i(t) &=\frac{1}{t-t_0} \int_{t_0}^t \sigma_i(s) d s,\\
u_i(t) &= \frac{1}{t-t_0} \int_{t_0}^t \mu_i(s) d s.
\end{aligned}
\end{equation}
Using definitions of $u_i$, $\mu_i$ and $a_i$ from Eqs. \ref{mu_diff}-\ref{au_def}, we can derive:
\begin{equation}
\label{Spiking LCA}
\dot{u}_i(t)=b_i(t)-u_i(t)-\sum_{j \neq i} w_{i j} a_j(t)-\frac{[u_i(t)-u_i\left(t_0\right)]}{t-t_0},
\end{equation}
which is the spiking analog of the original LCA dynamics (Eq. \ref{lca}). In the Spiking LCA, the potential accumulation is regulated by Eq. \ref{eqn:slca_v}. Consequently, the relationship between $u_i(t)$ and $a_i(t)$ satisfies  $T_\lambda\left(u_i(t)\right)-a_i(t) \rightarrow 0$ as $t \rightarrow \infty$, where $T_\lambda(\cdot)$ is described by

\begin{equation}
\label{Threshold}
T_{\lambda}(u(t))= \begin{cases}u(t)-\lambda & \text { if } u(t)>\lambda \\ 0 & \text { else.}\end{cases}
\end{equation}
Strict inhibitory connections ensure the average soma current remains within bounds. As $t \rightarrow \infty$, $\left(u_{i}(t)-u_{i}\left(t_{0}\right)\right) /\left(t-t_{0}\right) \rightarrow 0$. This indicates that the system tends towards the same limit as observed in LCA, which is equivalently the solution to the constrained LASSO problem Eq. \ref{CLASSO} \cite{tang2017sparse}.

\subsection{The generalized Spiking LCA}
The spiking LCA excludes essential features such as the leaky property and refractory period of a biological neuron. Furthermore, it considers a linear input-output curve, contrasting the non-linear dynamics observed in real biological neurons. Regarding the networks' architecture, existing LCA algorithms focus on inhibitory connections, limiting their application to a particular set of optimization problems. Therefore, to generalize the Spiking LCA to integrate a wide range of more biologically plausible neuron models in a unified framework \cite{tang2017sparse,davies2018loihi,fair2019sparse}, we now develop a generalized Spiking LCA.  


\begin{figure*}[!t]
\centering
\subfloat[]{\includegraphics[height=1.7in]{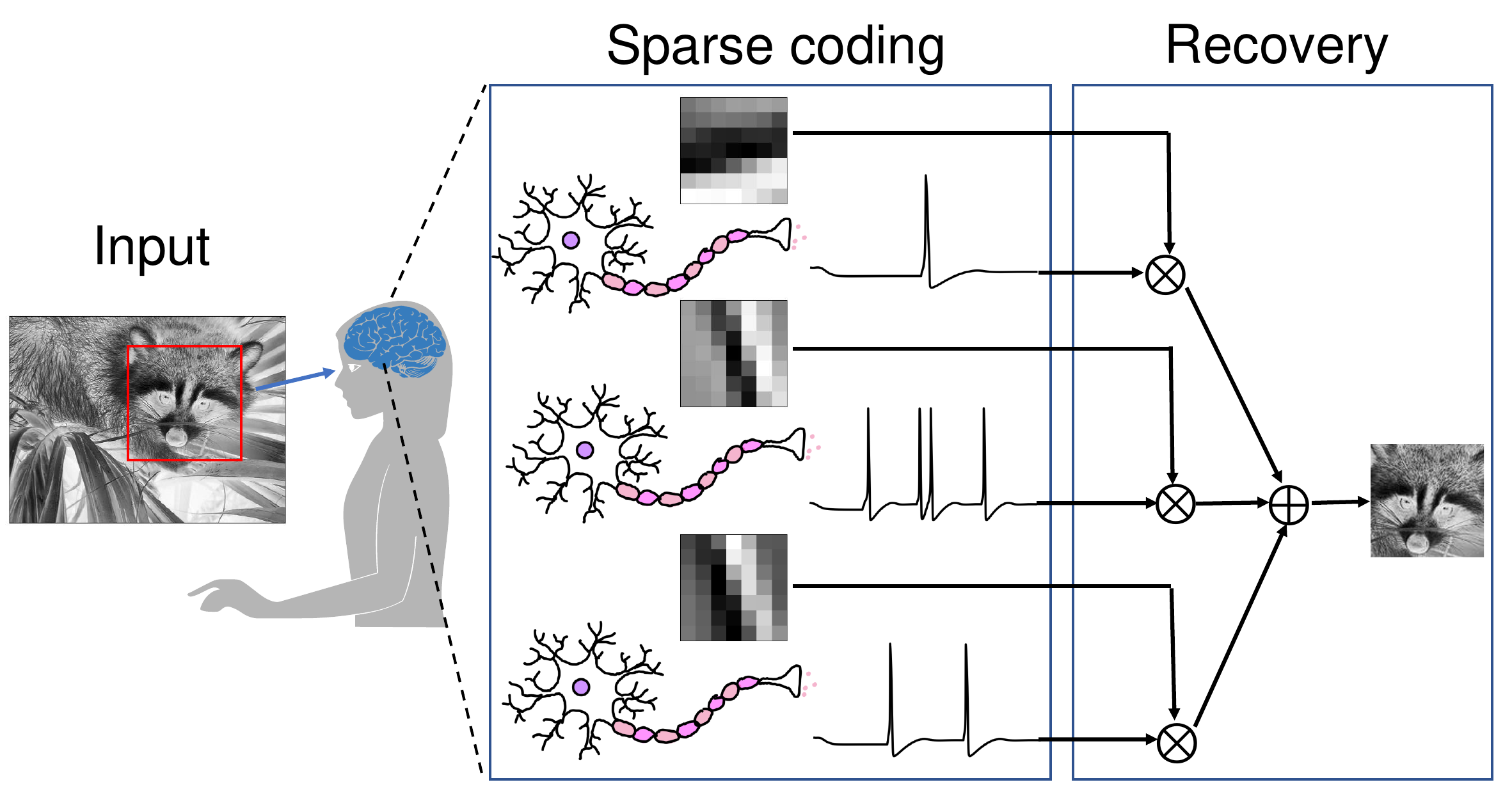}}
\hfil
\subfloat[]{\includegraphics[height=1.7in]{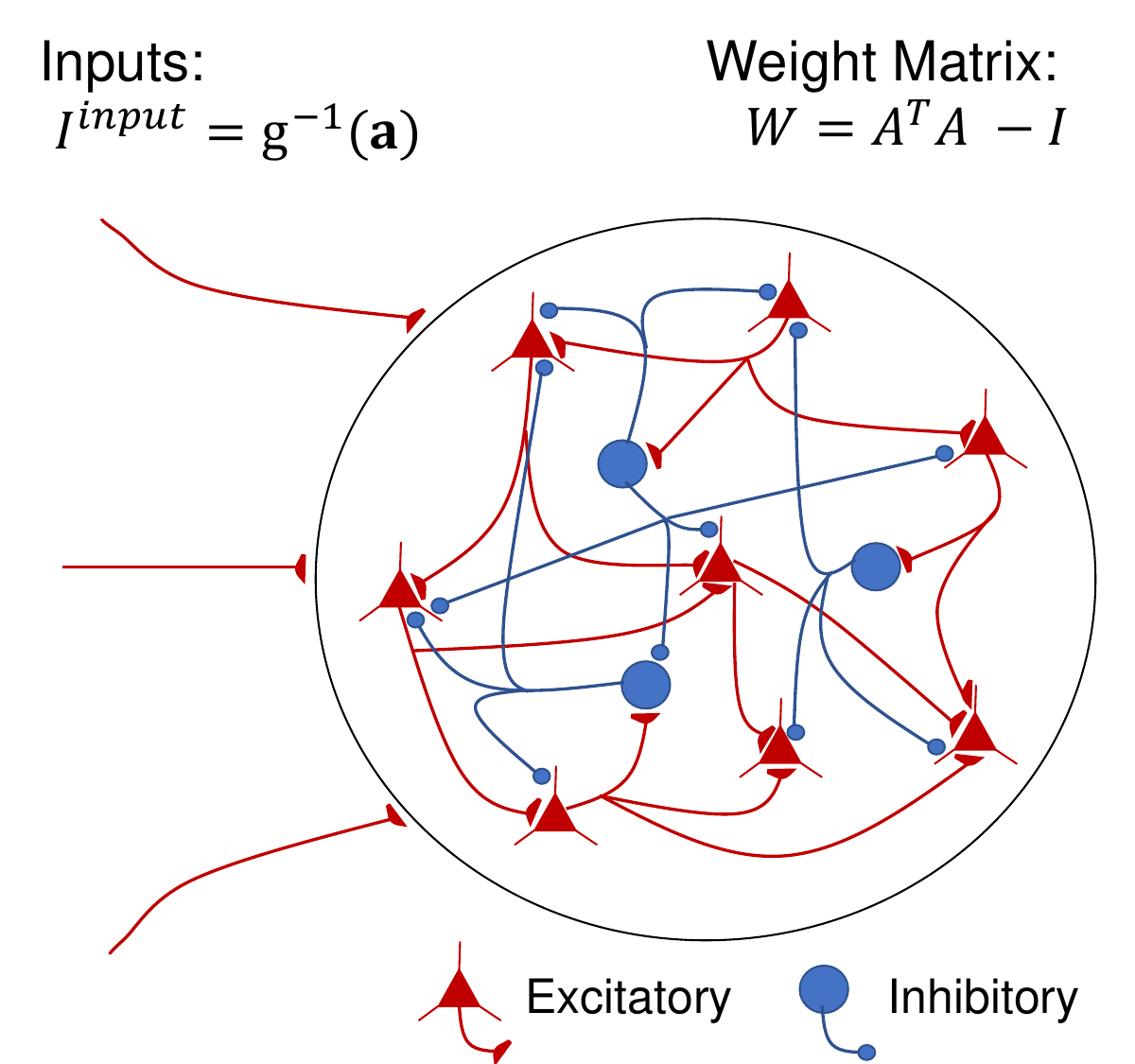}}
  \caption{The generalized Spiking LCA inspired by the visual cortex system. (a) Sparse coding is a technique used to simulate the sparse neural activity observed in the primary visual cortex. The input signal is reconstructed by computing a weighted sum of the receptive fields of model neurons, representing the specific regions of the input space that each neuron responds to. This approach allows for efficient and selective visual information processing, similar to the brain's.  (b) In the generalized Spiking LCA, each neuron receives an external input, $I^{input}=g^{-1}(\mathbf{a})$, as well as recurrent input from neighboring neurons.}
  \label{fig:concept}
\end{figure*}

Fig.~\ref{fig:concept} illustrates the sparse coding idea and the core architecture of the generalized Spiking LCA model. This model involves a network comprising $N$ interconnected neurons linked to all others through a current-based point neuron mechanism. The neuronal dynamics is governed by:

\begin{equation}
\label{lif}
\begin{aligned}
& c \frac{d v_{i}}{d t} = I^{ion}_{i}+I^{input}_{i}(t), \quad i=1, \ldots, N \\[10pt]
&\text{if } v_i\left(t\right) > v_{th} \quad  v_i\left(t\right) =v_{reset} \quad  t \in (t_{sp}, t_{sp} + t_{ref}),
\end{aligned}
\end{equation}
where $\mathrm{c}$ represents the neuron's membrane capacitance, $I^{ion}_{i}(t)$ is the ionic current in the neruon, and $I^{input}_{i}(t)$ refers to the injected current that depends on the recurrent inputs and the external constant inputs, which will be determined below. A neuron's membrane potential governs the generation of a spike train, which accumulates according to Eq. \ref{lif}. The corresponding neuron generates a spike when the membrane potential reaches the firing threshold $v_{th}$ at a specific time $t=t_{sp}$. This spike either inhibits or excites other neurons and resets its potential to the resting potential $v_{reset}$ during the refractory period. 

Following the spiking LCA, the average soma current $u_i(t)$ is required to follow
\begin{equation}
\label{gen Spiking LCA}
\dot{u}_i(t)=b_i(t)-u_i(t)-\sum_{j \neq i} w_{i j} a_j(t)-\frac{[u_i(t)-u_i\left(t_0\right)]}{t-t_0}.
\end{equation}
Note that, for different forms of ionic current $I^{ion}_{i}$, the relationship between firing rate  $a_i(t)$ and the input current $u_i(t)$ may not necessarily satisfy $a_i(t)=T_\lambda\left
(u_i(t)\right)$ as that in the classic Spiking LCA model, where $T_\lambda(u)$ is defined in Eq. $\ref{Threshold}$. As this relation is crucial to prove the convergence of the spiking LCA that solves the constrained Lasso problem, our generalized Spiking LCA requires further design on the input current $I_{i}^{\text{input}}$ to make the relation $a_i(t)=T_\lambda\left
(u_i(t)\right)$ hold. 


To design the input current, we next take the leaky integrate-and-fire model as an example, i.e., $I^{ion}_{i}=-g_{L}\left(v_{i}-v_{reset}\right)$, where $g_L$ is the leaky conductance, $v_{reset}$ is the reset potential after a spike. We can analytically solve the model and derive both the gain function $a=g(u)$ and its inverse function, which are depicted in Eq. \ref{gain_function}.

\begin{equation}
\label{gain_function}
\begin{aligned}
a = g(u) &=\left[t_{ref}- \frac{c}{g_L} \log \left(1-\frac{g_Lv_{th}}{u}\right)\right]^{-1}, u \ge g_L v_{th} \quad \\[10pt]
 \quad g^{-1}(a) &=\dfrac{g_Lv_{t h}}{1-\exp(g_L\left(t_{r e f}-\frac{1}{a}\right)/c) } .
\end{aligned}
\end{equation}
We compute the average soma current, $u_i(t)$, for each neuron at every time step based on Eq. \ref{gen Spiking LCA}. Using this current, we then calculate the input current $I_{i}^{\text{input}} = g^{-1}(T_{\lambda}(u_{i}))$ and apply it to each neuron in the subsequent time step. Accordingly, we then ensure that the output firing rate $a_{i}$ now satisfies the condition $a_{i}=T_{\lambda}\left(u_{i}\right)$. This holds because for the $i^{th}$ neuron,
\begin{equation}
    a_{i}= g(I_{i}^{\text{input}}) = g(g^{-1}(T_{\lambda}(u_{i}))) =T_{\lambda}\left(u_{i}\right).
\end{equation}
The above procedure can be generalized to a large variety of neuron models beyond the leaky integrate-and-fire model, and correspondingly we derive the following theorem for the convergence of our generalized Spiking LCA.



\begin{theorem}
    If the gain curve of the neuron model in the spiking neural network is continuous (not limited to the LIF model), then by applying an external input current $I_{i}^{\text {input }}=g^{-1}\left(T_{\lambda}\left(u_{i}\right)\right)$, as time approaches infinity, the firing rate 
$\mathbf{a}$ is equivalent to the output of LCA and converges to the optimal solution of Eq.~\ref{CLASSO}.
\end{theorem}

\begin{proof}
  For any neuron model with a continuous gain function  $g(\cdot)$, the average current dynamics of the $i^{th}$ neuron in the spiking neural network established based on this model satisfies 

\begin{equation}
\label{Theorem_1}
\begin{aligned}
\dot{u}_i(t) &=b_i(t)-u_i(t)-\sum_{j \neq i} w_{i j} a_j(t)-\frac{[u_i(t)-u_i\left(t_0\right)]}{t-t_0}, \\
a_{i} &=g\left(I_{i}^{\text {input }}\right).
\end{aligned}
\end{equation}
To prove the convergence of this system, we first introduce two lemmas:

\begin{lemma}
    If $g(\cdot) = T_\lambda(\cdot)$, applying any additional current at each step is unnecessary. As time approaches infinity, the firing rate $\mathbf{a}$ is equivalent to the output of LCA and converges to the optimal solution of Eq.~\ref{CLASSO}.
\end{lemma}
\begin{proof}
    see Ref.~\cite{shapero2014optimal} for details.
\end{proof}

\begin{lemma}
    There exists an upper bound $B_{+}$ and a lower bound $B_{-}$ such that $\mu_i(t), u_i(t) \in\left[B_{-}, B_{+}\right], \forall i, t \geq 0$.
\end{lemma}
\begin{proof}
In terms of network connections, our model distinguishes itself from previous works, which solely permitted inhibitory connections to maintain bounded soma current magnitudes and the corresponding average potentials. By incorporating realistic neuron models into our approach, the firing rate of the neurons is inherently limited, precluding it from becoming infinitely large. Consequently, we establish a lower bound and an $R>0$ such that $t_{i, k+1}-t_{i, k} \geq 1 / R$ for all $i=1,2, \ldots, n$, and $k \geq 0$, whenever two spike times are present. This insight confirms that the soma currents in our model are bounded both above and below. We define $C=\max_{i, j}\left|w_{i, j}\right|$ and $B=\max_j\left|b_j\right|$, acknowledging that the inner product of features and biases is finite. Employing the fact that $\left(\alpha * \sigma_j\right)(t) \leq \sum_{l=0}^{\infty} e^{-\frac{l}{R}}<\infty$, we demonstrate the following:

\begin{equation}
\begin{aligned}
\left\|\mu_i(t)\right\| & =\left\|b_i-\sum_{j \neq i} w_{i j}\left(\alpha * \sigma_j\right)(t)\right\| \\
& \leq
\left\|\left|b_i\right|+\sum_{
j \neq i}\left|w_{i j}\right|\left(\alpha * \sigma_j\right)(t)\right\| \\
& \leq\left\|\max _j\left|b_j\right|+\sum_{j \neq i}\left|w_{i j}\right|\left(\alpha * \sigma_j\right)(t)\right\| \\
& \leq\left\|B+n C\left
(\alpha * \sigma_j\right)(t)\right\| \\
& \leq\left\|B+n C \sum_{l=0}^{\infty} e^{-\frac{l}{R}}\right\|<\infty.
\end{aligned}
\end{equation} 
Implying the soma currents are bounded from above and below.
\end{proof}

 Hence, we adopt the proof of Ref. \cite{tang2017sparse} and state $u(t)=\left[u_1(t), u_2(t), \ldots, u_N(t)\right]^T$ has at least one limit point $u^* \in \mathbb{R}^N$ such that $u\left(t_k\right) \rightarrow u^*$ as the sequence of $t_k  \rightarrow \infty$ when $k \rightarrow \infty$ from the Bolzano-Weirstrass theorem. This implies:

\begin{equation}
\label{aike}
\lim _{t \rightarrow \infty} \dot{u}_i(t)=\lim _{t \rightarrow \infty} \frac{\mu_i-u_i}{t-t_0}=0.
\end{equation}
For other neuron models, we evaluate the average soma current $u_i(t)$ for each neuron at each iteration. We then calculate the input current $I_{i}^{\text {input }}=g^{-1}\left(T_{\lambda}\left(u_{i}\right)\right)$ using this current. Subsequently, we determine the output firing rate $a_{i}$ as $a_{i}=T_{\lambda}\left(u_{i}\right)$ by applying the activation function $g(\cdot)$ to input current $I_{i}^{\text {input }}$. This sequence of steps ensures that our spiking neural network converges to the solution of the constrained LASSO problem. The dynamics of $u_i(t)$ can be expressed as follows:

\begin{equation}
\label{Vars_spiking_LCA}
\begin{aligned}
\dot{u}_{i}(t) &=b_{i}-u_{i}(t)-\sum_{j \neq i} w_{i j} a_{j}(t), 
\\
a_{i} &=g\left(I_{i}^{\text {input }}\right)=T_{\lambda}\left(u_{i}\right) \quad t \rightarrow \infty.
\end{aligned}
\end{equation}
Hence $T_\lambda \left(u\left(t_k\right)\right) \rightarrow T_\lambda\left(u^*\right)=a^*$, we can conclude the system converges to the same limit found in LCA.  With the above results, we complete the proof.
\end{proof}

In the Appendix,   four commonly used neruonal models are introduced. Note that analytical expressions for the gain function of these models are generally infeasible. However, we can approximate the gain function numerically, and subsequently incorporate them into our algorithm and perform numerical experiments below.

\section{Numerical experiments}

This section presents a series of numerical experiments to demonstrate the effectiveness of the generalized Spiking LCA and its convergence behavior.  Our tests encompass both synthetic and real datasets. For the synthetic data, we aim to confirm the solution equivalence between the generalized Spiking LCA and various optimization problems. Subsequently, we juxtapose the performance of our algorithm with different penalty functions, highlighting the superiority of our algorithm regarding power consumption and processing time. The dictionary entries  $\Phi$ are sampled randomly from a standard Gaussian distribution, represented by $\mathcal{N}(0,1)$
. Additionally, we explore the algorithm's efficacy in several practical applications, such as sparse signal and CT image reconstruction. All numerical experiments were conducted the Brainpy neural engineering simulation platform \cite{wang2023brainpy} on a server powered by an A100 GPU platform. Recent work indicates that our algorithm can be implemented on the neuromorphic chip Loihi \cite{davies2018loihi,zhang2022spiking}.

\subsection{Signal recovery}
\label
{sec:signal_recovery}
\begin{figure*}[!t]
\centering
\subfloat[]{\includegraphics[width=2.6in]{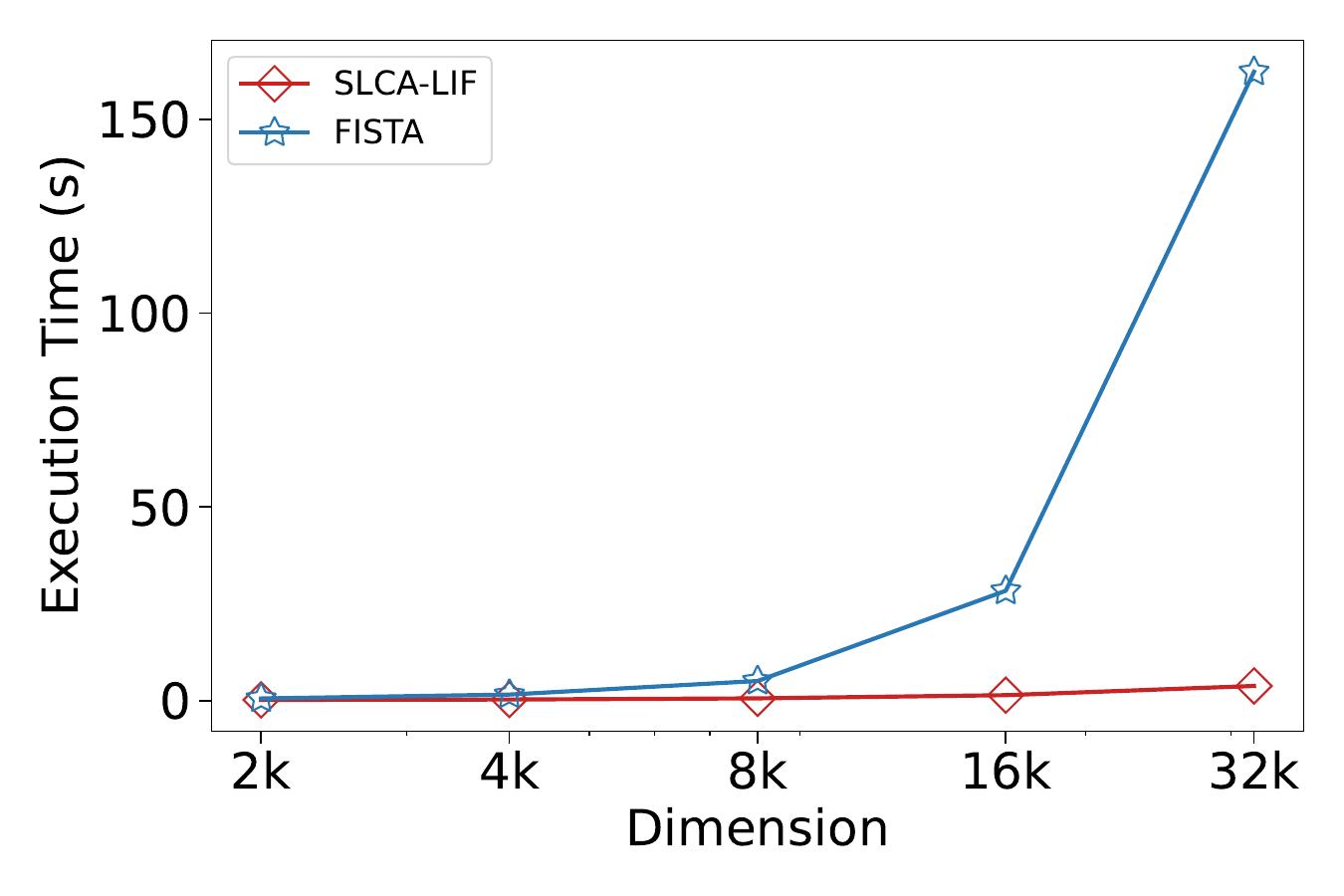}}
\hfil
\subfloat[]{\includegraphics
[width=2.6in]{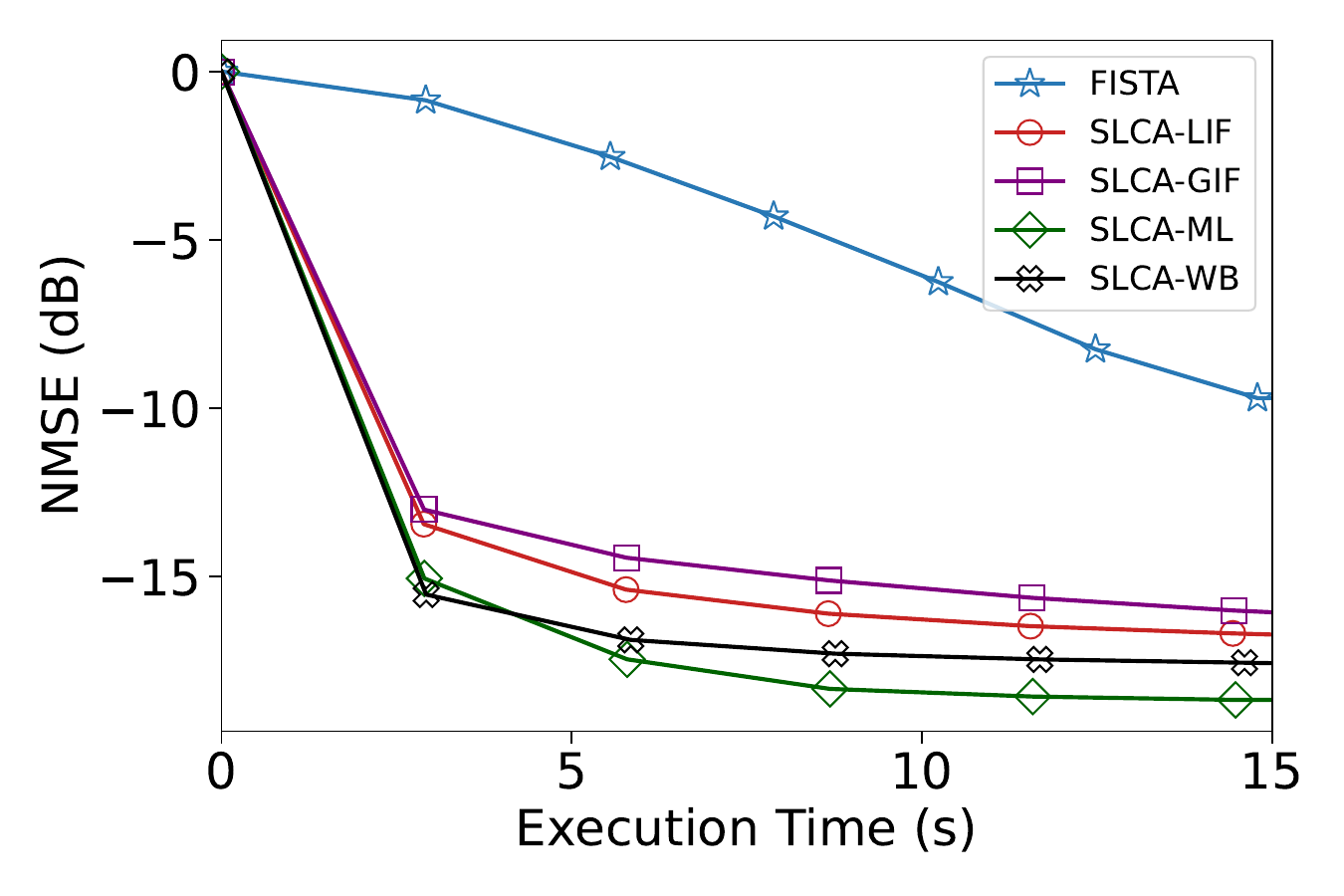}}
  \caption{Numerical results for the toy model. (a) The computational time for solving LASSO problems of varying sizes, comparing spiking neural networks with FISTA. (b) Performance of spiking neural networks based on various models, i.e., Leaky Integrate-and-Fire, Generalized Integrate-and-Fire model, Morris-Lecar, and Wang-Buzsaki model, in solving optimization problems, $A \in \mathbb{R
}^{5000 \times 10000}$. The metric $N M S E=10 \log _{10}\left(\|\hat{a}-a\|_2^2 /\|\hat{a}\|_2^2\right)$ indicates normalized mean square
error, where $\hat{a}$  denotes the original signal.}
  \label{fig:different_neurons}
\end{figure*}

In the last decade, sparse signal recovery, particularly in solving the LASSO problem, has attracted significant attention from researchers. Hence, to assess our algorithm's performance, we start with a toy model involving the recovery of a real-valued signal, denoted as $\mathbf{a} \in \mathbb{R}^{N}$, where $N$ is the signal length. We generate an observation matrix $A \in \mathbb{R}^{m \times n}
$ (with $m < n$), where each entry follows an independent Gaussian distribution. The observation vector $\mathbf{b}$ is the product of $A$ and $\mathbf{a}$, i.e., $\mathbf{b} = A\mathbf{a}$. In this scenario, $\mathbf{a}$ is a sparse vector with sparsity level $K$, meaning it has $K$ nonzero elements ($||\mathbf{a}||_0=K$).  Our experiments assess the execution time of FISTA and generalized Spiking LCA, and their recovery performance. As the signal recovery problem scale enlarges, FISTA's runtime significantly increases. In contrast, generalized Spiking LCA demonstrates a steadier ascent in computation time, as illustrated in Fig.~\ref{fig:different_neurons}(a). This difference is due to our algorithm's execution on the A100 GPU, which supports extensive parallel computations. As a result, the growth in problem size does not lead to a sharp rise in the computation time of our method. 

In principle, the execution time of the generalized spiking LCA can be further substantially reduced if we implement the algorithm on the neuromorphic chip like Loihi by achieving parallel computation
across all neuron nodes. In this context, Loihi dedicates additional resource cores to manage larger neuron sizes, enabling extensive parallel processing across all cores. Hence, the runtime of generalized Spiking LCA is more influenced by factors such as the number of neurons within a resource core and spike traffic rather than the specific scale of the optimization problem. This feature highlights the substantial potential of generalized Spiking LCA for practical applications and its promising advantages in real-world scenarios.

We next evaluate the performance of the generalized Spiking LCA. Apart from the LIF model, our algorithm can construct networks based on other biophysical neuron models, such as the GIF model \cite{mihalacs2009generalized}, the Morris-Lecar (M-L) model \cite{tsumoto2006bifurcations}, and the Wang-Buzsaki model \cite{wang1996gamma}. The detailed model description can be found in the Appendix. 
For the ease of comparative testing,  we set $n=10000$, $m=5000$, and $K=500$, and use normalized mean square error (NMSE) as a measure of estimation error. The simulation results are illustrated in Fig.~\ref{fig:different_neurons}(b), which infers that our algorithm has a quicker initial convergence than the FISTA method across diverse SNN architectures with different neuronal models.

\begin{figure*}
\centering
\subfloat[]{\includegraphics[width=2.6in]{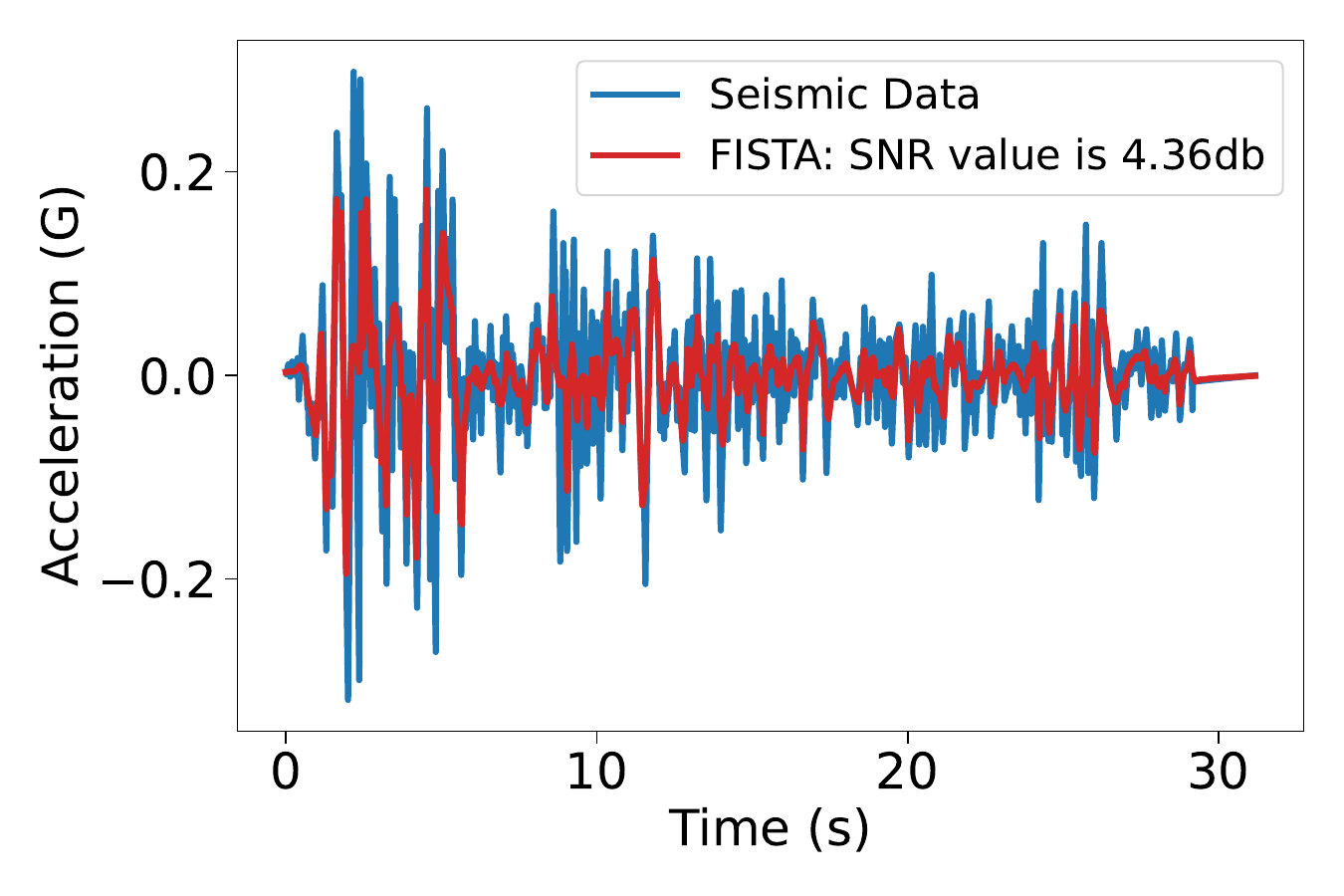}}
\hfil
\subfloat[]{\includegraphics[width=2.6in]{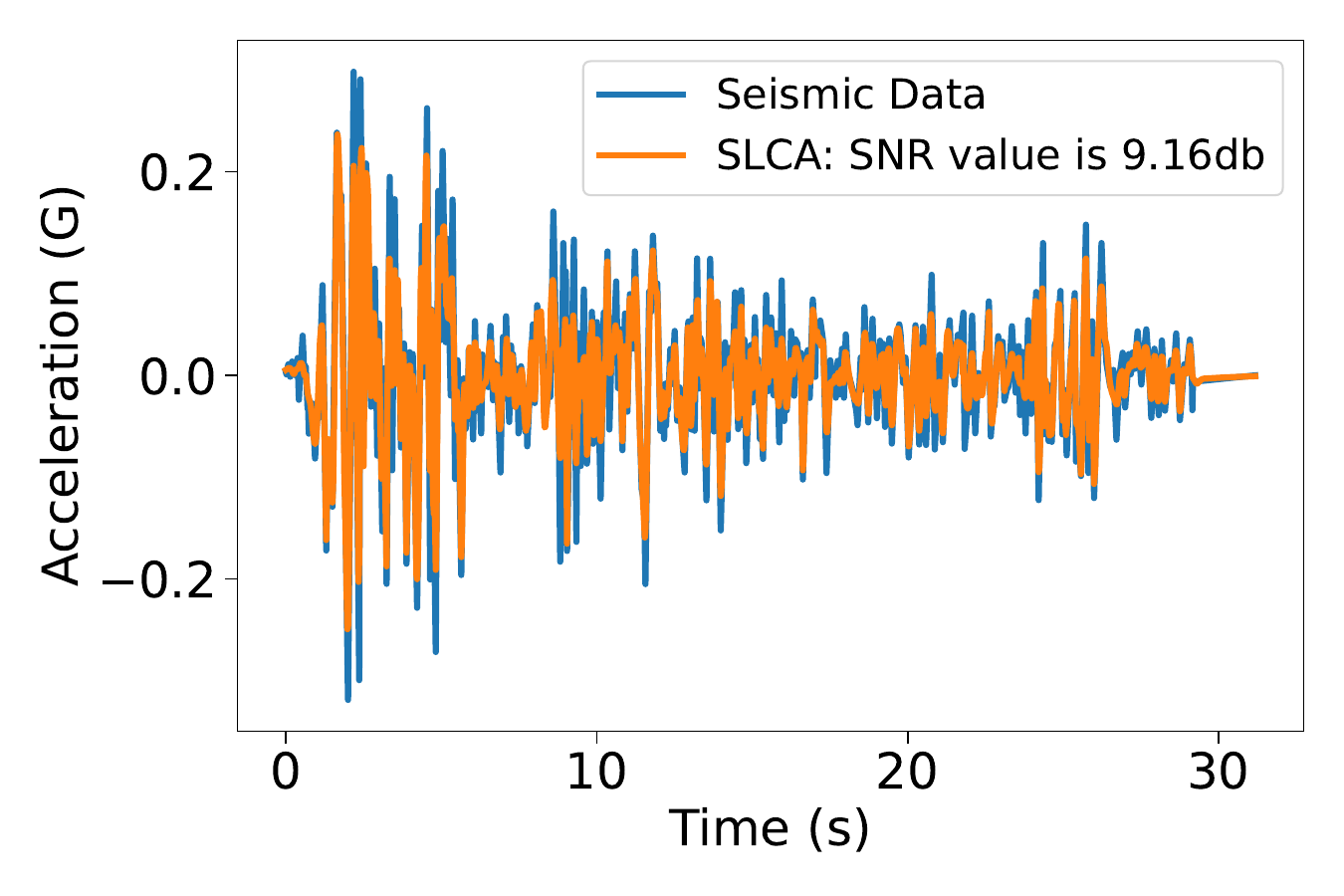}}
\caption{Comparison of seismic signal reconstruction performance using the FISTA and generalized Spiking LCA algorithms. (a) Reconstruction results of the FISTA algorithm, where the red line represents the reconstructed signal and the blue line represents the original data. (b) Reconstruction results of the generalized Spiking LCA algorithm, where the yellow line represents the reconstructed signal and the blue line represents the original data.}
\label{fig:seismic}

\end{figure*}

Next, we demonstrate the utility of our method in a more practical context: the reconstruction of seismic wave signals, a crucial aspect of the Earth's subsurface exploration and monitoring. Our experiment focuses on reconstructing seismic wave signals using the Ricker wavelet, renowned for its ability to provide a sparse representation of these signals.  It showcases a relatively straightforward appearance characterized by a dominant positive peak flanked by two negative side lobes. From an analytical perspective, its representation in the time domain is as follows:
\begin{equation}
\label{eq:ricker}
A(t)=\Big[1-2 \pi^2 f^2 (t-t_0)^2\Big] e^{-\pi^2 f^2 (t-t_0)^2}.
\end{equation}
Here, $A(t)$ delineates the wavelet's amplitude at time $t, f$ denotes the dominant frequency, and $t_0$ represents the center time of the wavelet. The goal is to recover the wavelet coefficient sequence $\mathbf{a}$ via the following model:

\begin{equation}
\label{Seismic_wave}
\min _{\mathbf{a} \ge 0} \frac{1}{2}\|\mathbf{s}-A \mathbf{a}\|_2^2+\lambda \|\mathbf{a}\|_1,
\end{equation}
where $\mathbf{s}$ is the measured seismic data, $A$ is the Ricker wavelet matrix, and $\lambda$ controls the trade-off between the data fitting and the sparsity terms. By adjusting the frequency and center time of Ricker wavelets, you can create a diverse set of wavelets, each with its own distinctive oscillation speed and temporal positioning. The matrix $A$ is then formed by sampling these varied wavelets at specific time points, with each column capturing the sampled values of a wavelet at a given frequency and center time, thereby encompassing a broad spectrum of seismic characteristics. Utilizing $A$, we model $\mathbf{s}$ by identifying an optimal combination of wavelet coefficients in the sequence $\mathbf{a}$. This approach enhances the interpretation of seismic data, improves the detection of seismic events, and facilitates the inversion process for estimating subsurface characteristics. 

We evaluate the performance of both FISTA and the generalized Spiking LCA in recovering this signal, specifically using the El Centro Earthquake dataset. The effectiveness of each algorithm is assessed using the signal-to-noise ratio (SNR). As shown in  Fig.~\ref{fig:seismic}, upon execution for a consistent timeframe of 2 seconds, the generalized Spiking LCA algorithm is demonstrated to have superior signal reconstruction quality compared to the FISTA algorithm, achieving an SNR value of 9.16 dB, in contrast to FISTA’s 4.36 dB.

\subsection{Computed Tomography construction}
\label{sec:CT}
\begin{figure*}[!t]
\centering
\subfloat[]{\includegraphics[width=1.75in]{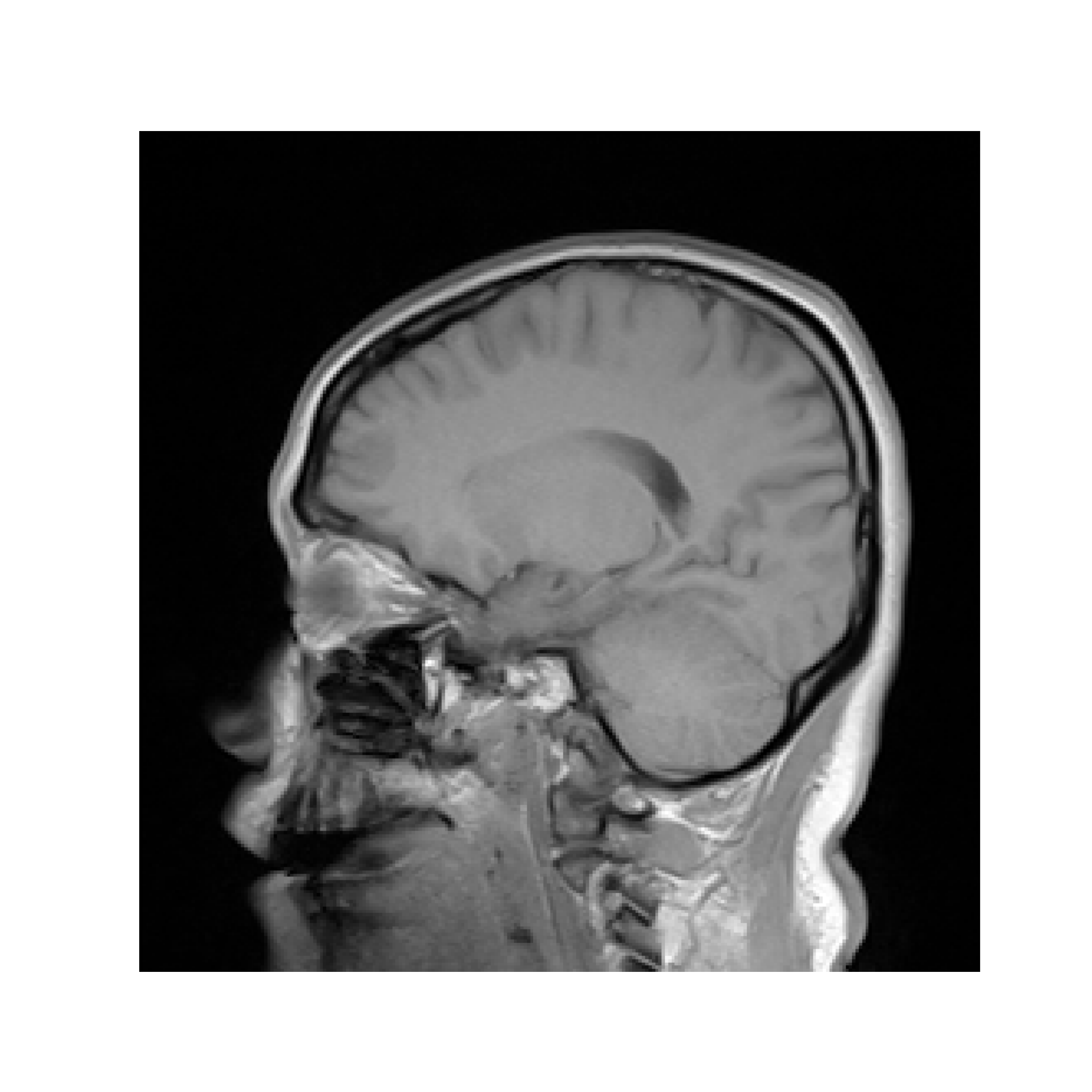}}
\hfil
\subfloat[]{\includegraphics[width=1.75in]{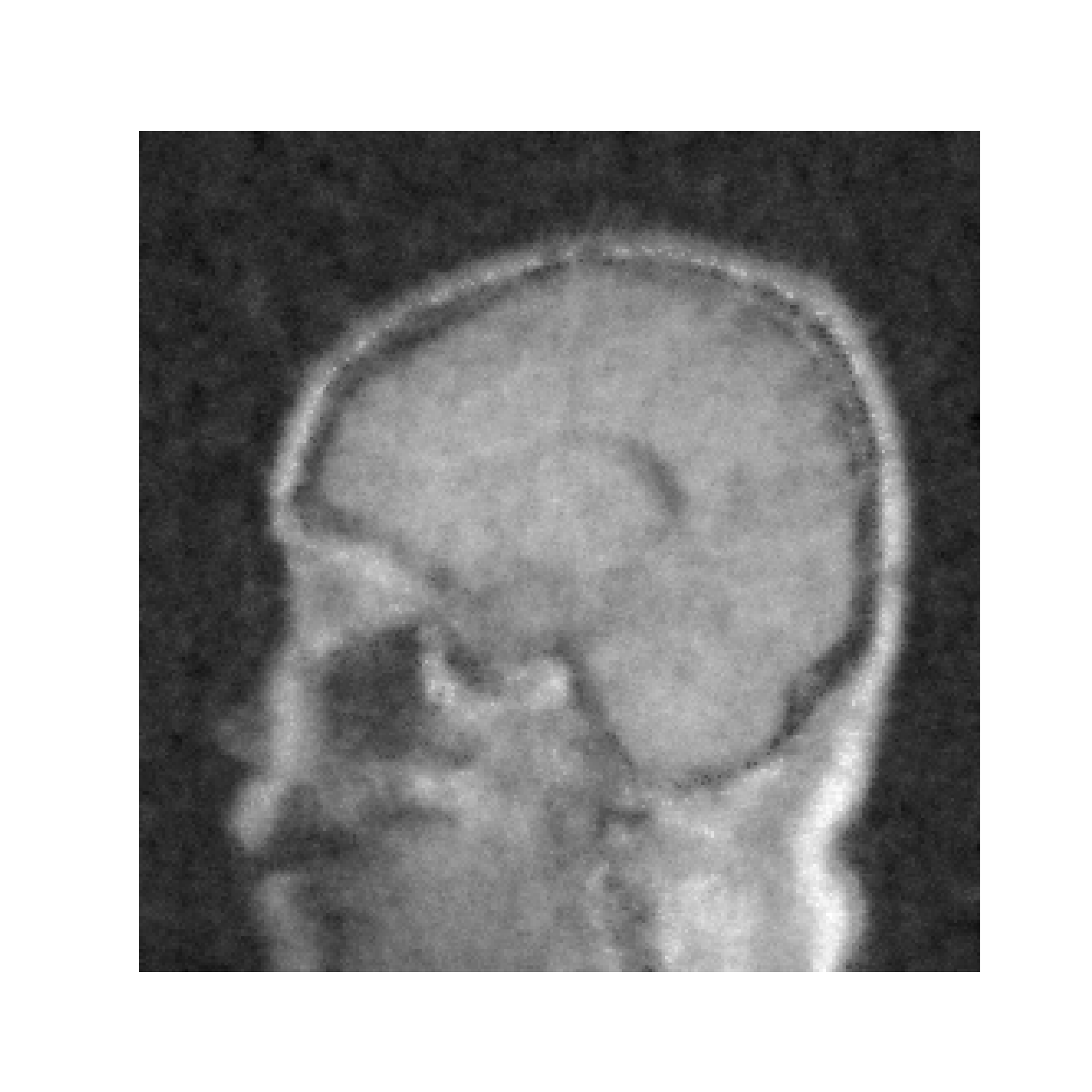}}
\hfil
\subfloat[]{\includegraphics[width=1.75in]{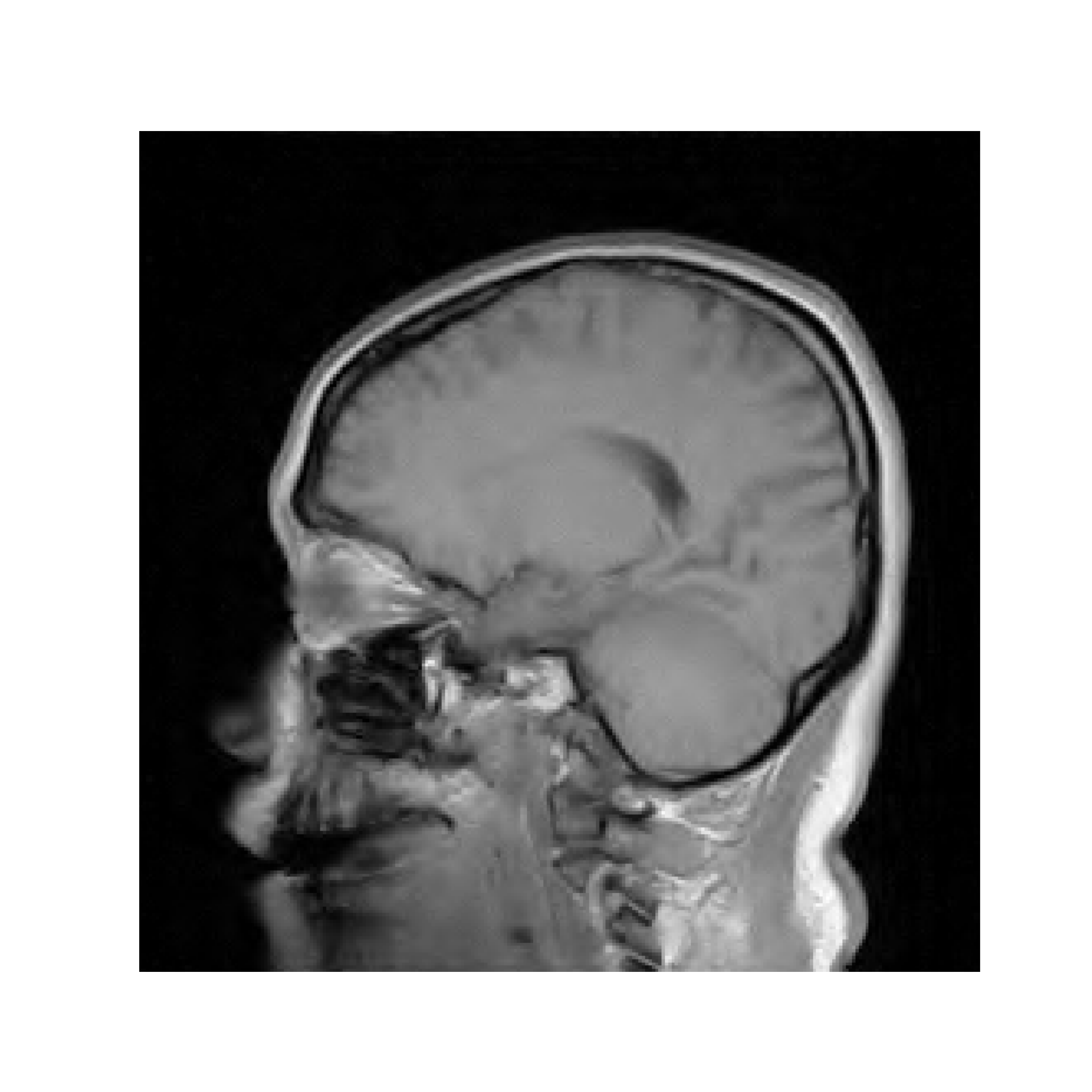}}
  \caption{Reconstruction of $256 \times 256$ pixel CT images from simulated CS acquisition. (a) The original image. (b) The reconstructed image by FISTA, with PSNR value 24.4 dB at 12s. (c) The reconstructed image by the generalized Spiking LCA, with PSNR value 34.6 dB at 12s. These comparative images demonstrate the superior performance of the generalized Spiking LCA algorithm over FISTA when both algorithms are run for the same time.}
  \label{fig:CT_reconstruction}
\end{figure*}

This section highlights our system's potential benefits in a medical imaging application, where real-time compressive sensing (CS) recovery techniques can provide substantial improvements. Computed tomography (CT) is a widely used imaging technique in medical diagnosis and treatment. It involves using X-rays to create detailed images of the body's internal structures. Compressive CT imaging is of significant interest because it can reduce scan times, improving patient throughput and safety by reducing radiation exposure.  

In this study, we explore the effectiveness of two different algorithms, generalized Spiking LCA and FISTA, for reconstructing a head CT image. The image can be represented as a vector $\mathbf{x}$, which is sparse in the Daubechies  wavelet basis, i.e., it can be represented using a small number of wavelet coefficients $\mathbf{x} = \Psi \mathbf{a}$, where $\Psi$ is the discrete orthogonal  wavelet transform matrix \cite{shen1998asymptotics}. The signal $\mathbf{x}$ can be measured using an $M \times N$ discrete Gaussian random measurement matrix $\Phi$, with $M \ll N$ (e.g., the size ratio $M/N = 0.2$), resulting in a compressed measurement vector $s \in \mathbb{R}^M$. The experiment's objective is to recover the head CT image $\mathbf{x}$ using the following model:

\begin{equation}
\label{CT_reconstruction}
\min _{\mathbf{a}}  \frac{1}{2}\|s-\Phi \Psi \mathbf{a}\|_2^2+\lambda 
\|\mathbf{a}\|_1
\end{equation}
Note that in this problem, the coefficient $\mathbf{a}$ is not restricted to $\mathbf{a}\ge 0$. Our algorithm can handle this problem as well with a modification in the procedure. Details of these modifications will be discussed in later section.

We evaluate the performance of the generalized Spiking LCA and FISTA algorithms for image reconstruction in CT using peak signal-to-noise ratio (PSNR) as a measure of estimation error. Fig.~\ref{fig:CT_reconstruction} shows that the generalized Spiking LCA algorithm can produce higher quality reconstructed images, while the quality of images obtained by FISTA is relatively poor.


\subsection{Extensions}
In our previous experiments, we set  $\widetilde{C}(\mathbf{a}) = \|\mathbf{a}\|_1$
in Eq. \ref{origianl_function} to solve the LASSO problem.
As we progress, we intend to show the versatility of our algorithmic framework in catering to diverse optimization problems. This adaptability can be achieved by substituting  $\widetilde{C}(\mathbf{a})$ with alternative penalty functions. Specifically, we investigate our algorithm's capability to tackle  the Elastic-Net optimization and the unconstrained LASSO problem, which are two important problems often encountered in real-world applications.  In fact, for any penalty function that satisfies the following rules,
\begin{enumerate}
    \item $\widetilde{C}(\cdot)$ is non-negative and subanalytic on $[0,+\infty)$.
    \item The first-order derivative of  $\widetilde{C}(\cdot)$ is continuous and non-negative on $[0,+\infty
)$, i.e., $\widetilde{C}^{\prime}(\cdot) > 0$.
    \item  Define $T_\lambda^{-1}(a) = \lambda \frac{d \widetilde{C}(a)}{d a} + a$, then the first-order derivative of  $T_\lambda(
\cdot)$ is continuous and positive on $(0,+\infty
)$. i.e., $T^{\prime}_\lambda(\cdot) > 0$.
\end{enumerate}
For rules 1 and 2, convergence of the corresponding LCA system is guaranteed~\cite{balavoine2013convergence,chen2014convergence,fosson2018biconvex,zhang2022spiking}. Regarding rule 3, it is established that within the interval $(0,+\infty)$, functions such as $T_\lambda(\cdot)$ and $T^{-1}_\lambda(\cdot)$ exist.  Based on these assurances, we can establish the following theorem:


\begin{theorem}
    Consider $\widetilde{C}(\cdot)$ satisfying  previously mentioned rules 1-3. Then, the firing rate $\mathbf{a}$ of the spiking neural network is globally asymptotically convergent, and $\mathbf{a}$ will converge to the solution of the corresponding optimization problem described by Eq. \ref{origianl_function}.
\end{theorem}
\begin{proof}
    Let $\widetilde{C}(\cdot)$ adhere to all the stipulations previously outlined. Under such a condition,  the average current dynamics of the $i^{th}$ neuron in the spiking neural network  satisfies 

\begin{equation}
\label{Theorem_2}
\begin{aligned}
\dot{u}_i(t) &=b_i(t)-u_i(t)-\sum_{j \neq i} w_{i j} a_j(t)-\frac{[u_i(t)-u_i\left(t_0\right)]}{t-t_0}, \\
a_{i} &= T_\lambda(u_i), \quad 
\lambda \frac{d \widetilde{C}\left(a_i\right)}{d a_i}=u_i-a_i=T_\lambda^{-1}\left(a_i\right)-a_i ,
\end{aligned}
\end{equation}
where the interconnection between $\widetilde{C}(a_i)$ and $T_\lambda(u_i)$ is given by the equation $\lambda \frac
{d \widetilde{C}(a_i)}{d a_i} = T_\lambda^{-1}(a_i) - a_i$. By utilizing this relationship, we can solve for $a_i$ in terms of $u_i$ to determine the corresponding activation function $T_\lambda(\cdot)$.

To tackle the  optimization problem by Eq. \ref{origianl_function}, we perform the following steps for each neuron in our network. First, we compute the average current $u_i(t)$ at every iteration step. Then, we inject the current $I_{i}^{\text{input}} = g^{-1}\left(T_{\lambda}\left(u_{i}\right)\right)$, where different activation functions $T_{\lambda}(\cdot)$ are used for the optimization problem with different penalty functions. This leads to the neuron's output activation $a_i = g\left(I_{i}^{\text{input}}\right) = g\left(g^{-1}\left(T_{\lambda}\left(u_{i}\right)\right)\right) = T_{\lambda}\left(u_{i}\right)$. This procedure ensures the convergence of our spiking neural network towards the solution of the general optimization problem (Eq. \ref{origianl_function}).
\end{proof}

For some penalty function $\widetilde{C}(\cdot)$, the expression for the activation functions $T_\lambda(\cdot)$ can be quite complex, making the numerical implementation of the generalized Spiking LCA inefficient. To solve this issue, we adapt an alternative strategy where we do not solve $T_\lambda(\cdot)$ explicitly. Instead, we inject the input current as $I_i^{\text {input }}=g^{-1}\left(u_i - \lambda \frac{d \widetilde{C}\left(a_i\right)}{d a_i}  \right)$. This approach can also lead to the neuron's output activation 
\begin{equation*}
   a_i=g\left(I_i^{\text {input }}\right)=g\left(g^{-1}\left(u_i -\lambda \frac{d \widetilde{C}\left(a_i\right)}{d a_i}\right)\right)= u_i -\lambda \frac{d \widetilde{C}\left(a_i\right)}{d a_i} = T_\lambda\left(u_i\right). 
\end{equation*}
The procedure also guarantees that the spiking neural network converges to the solution of the general optimization problem of Eq. \ref{origianl_function}.

\subsubsection{Elastic-Net}
\begin{figure*}[!t]
\centering
\subfloat[]{\includegraphics[width=2.6in]{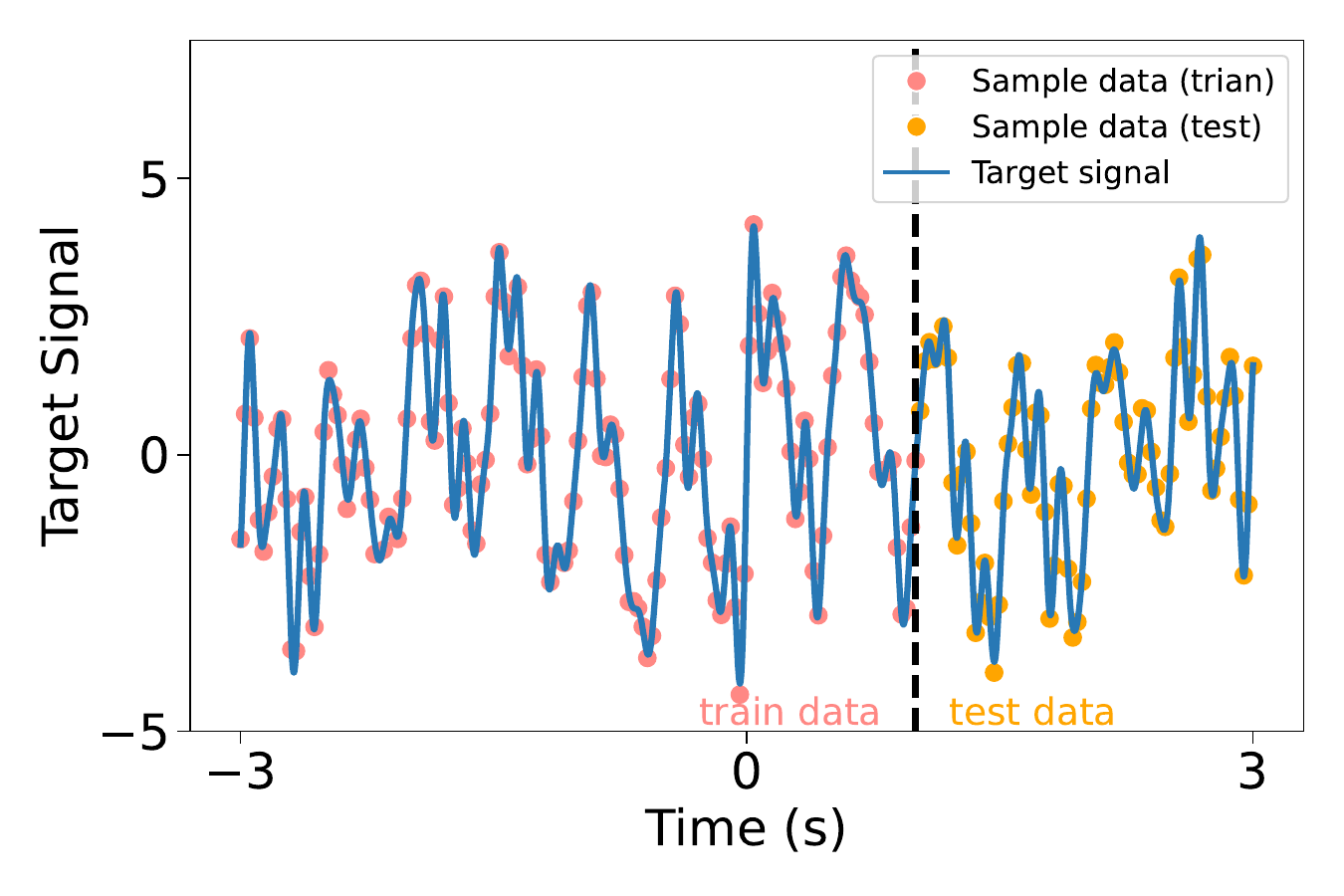}}
\hfil
\subfloat[]{\includegraphics[width=2.6in]{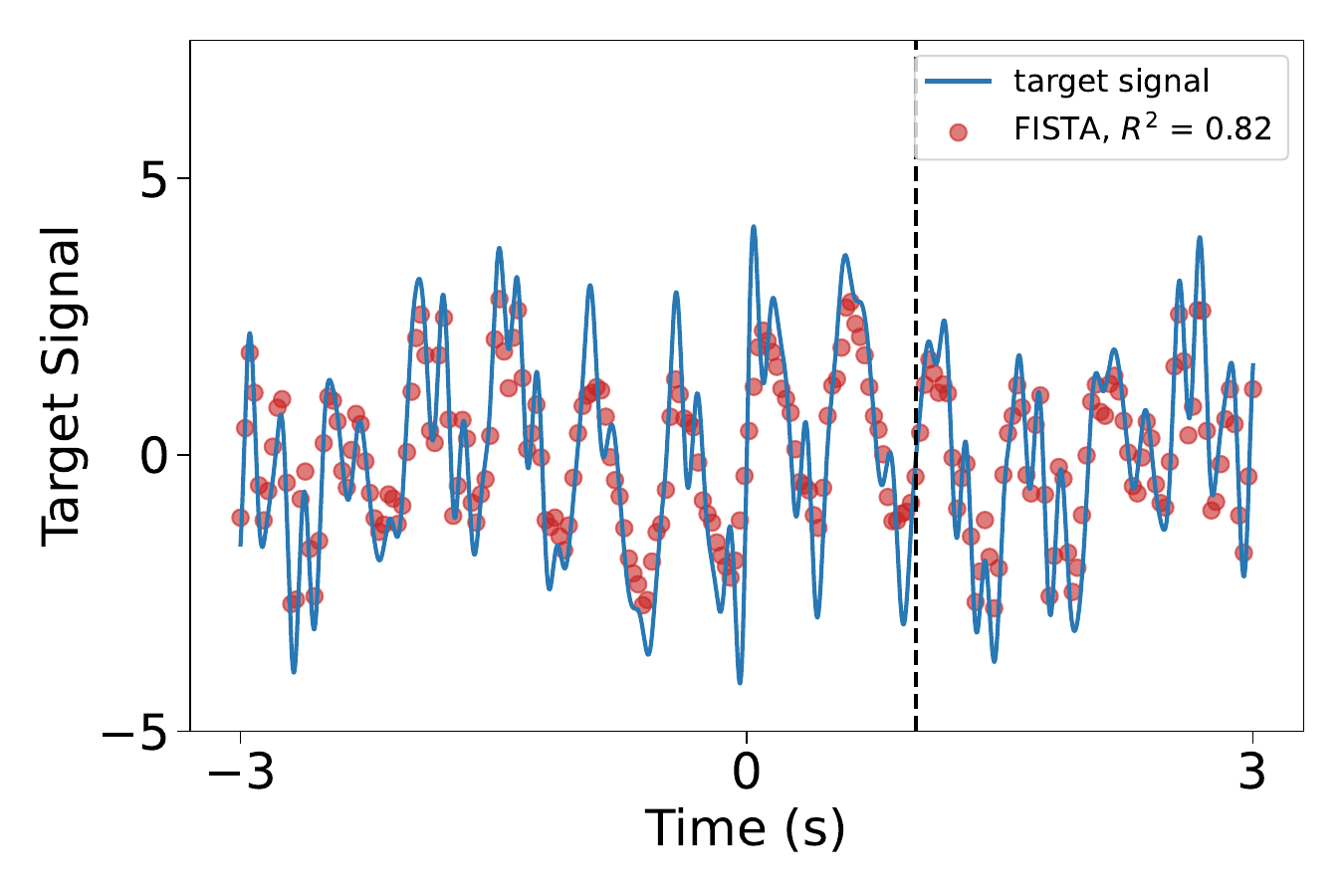}}
\\
\subfloat[]{\includegraphics[width=2.6in]{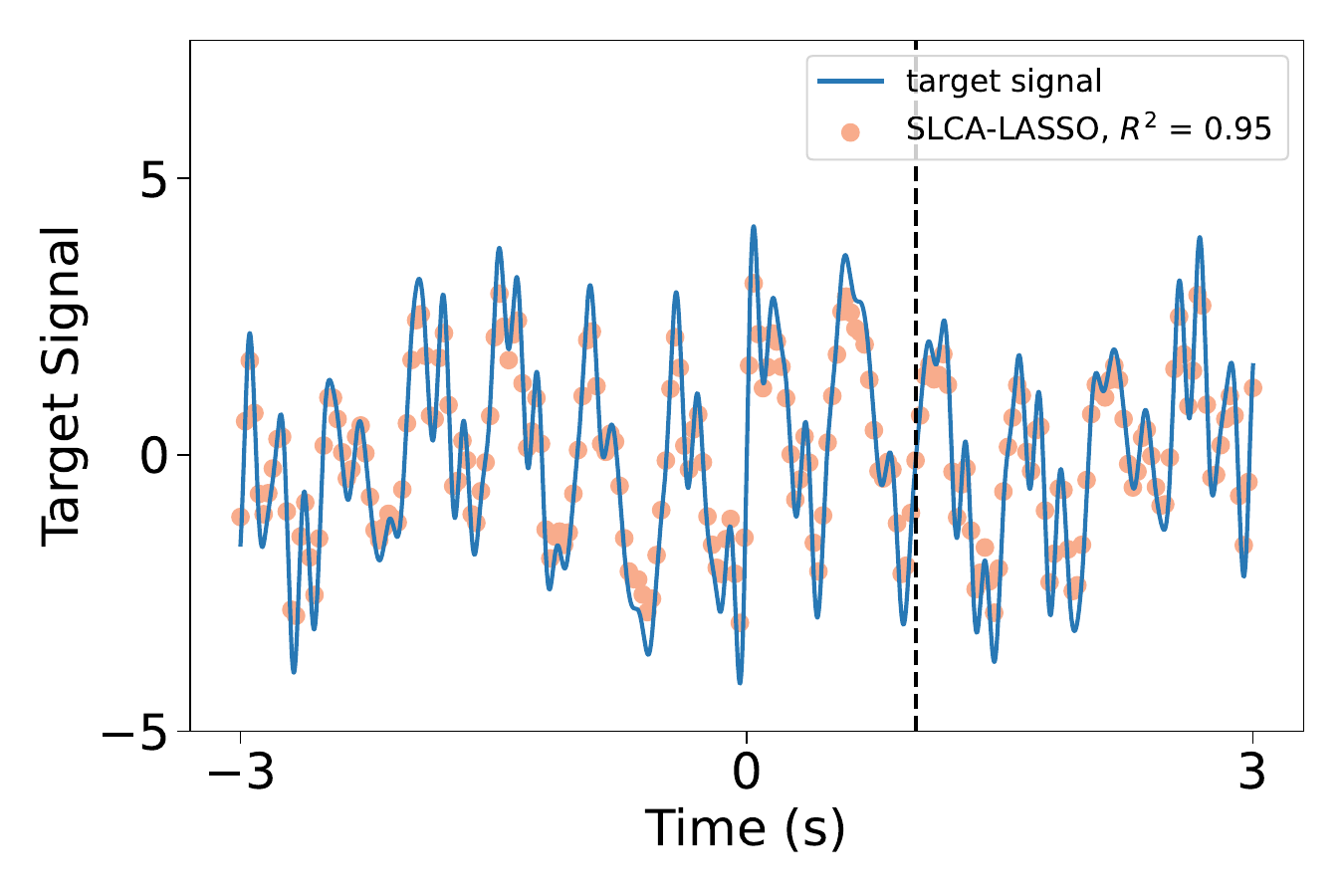}}
\hfil
\subfloat[]{\includegraphics[width=2.6in]{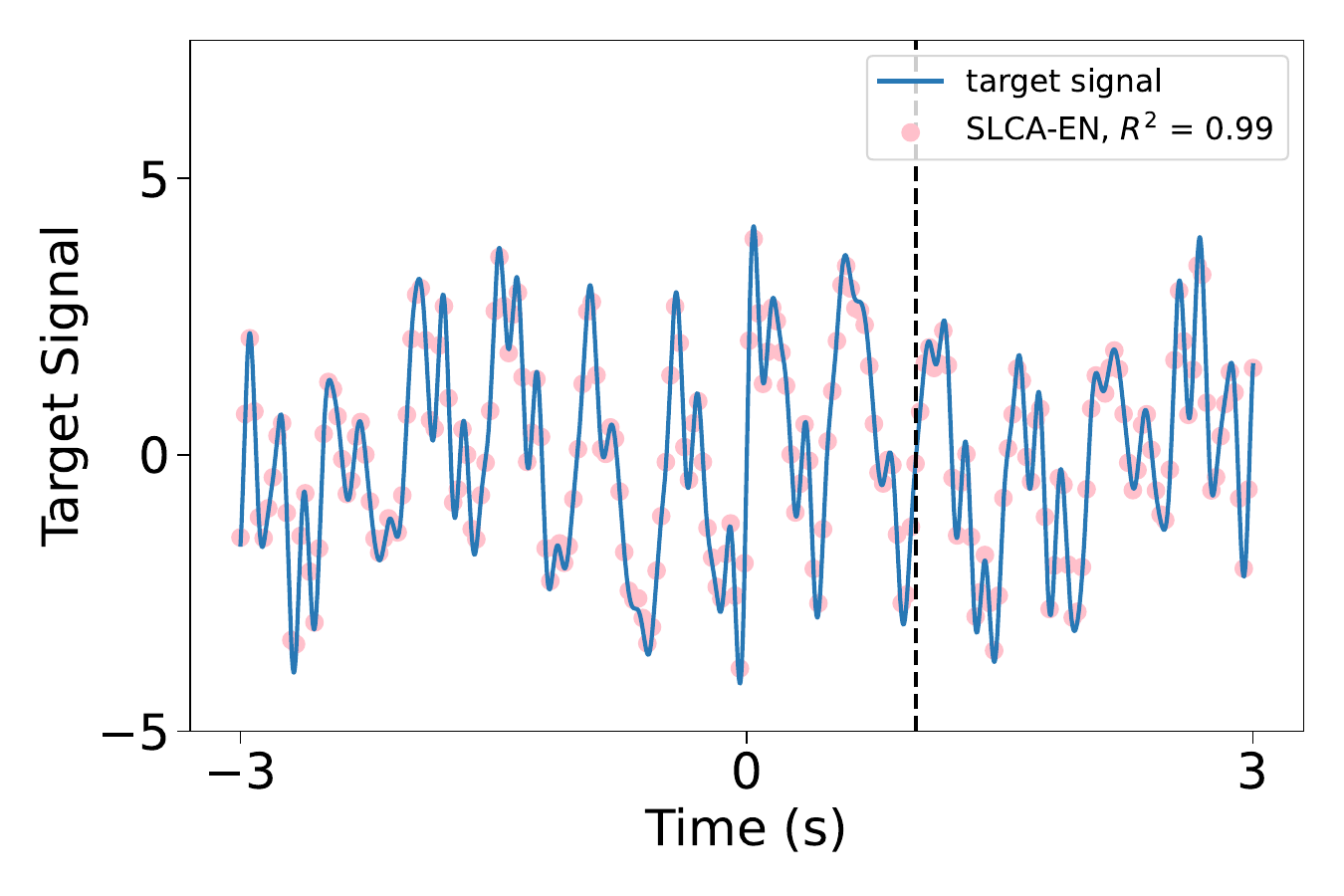}}
\caption{LASSO and Elastic-Net methods for sparse signals recovery. (a)  We partition the sampled points into training and testing sets to restore superimposed sinusoidal signals. (b) Reconstruction results of the FISTA algorithm, where the red pot represents the reconstructed signal and the blue line represents the  target signal. The $R^2$ value is used to measure the discrepancy between the recovered signals in the test set and the actual data points. The execution time for the algorithm is standardized at 2 seconds, during which the FISTA optimization algorithm achieves an $R^2$ value of 0.82.  (c) The $R^2$ value for LASSO is 0.95. (d)  Elastic-Net achieves an $R^2$ value of 0.99. 
}
\label{fig:Elastic_Net}
\end{figure*}

The Elastic-Net combines the $\ell^{1}$ and $\ell^{2}$ penalties of the LASSO and Ridge methods in a unified regularization approach  \cite{zou2005regularization}, described as:
\begin{equation}
\label{elastic_net}
   \min _{\mathbf{a} \geq 0} \frac{1}{2}\|\mathbf{s} -\Phi \mathbf{a}\|_2^2+\lambda(\rho\|\mathbf{a}\|_1+\frac{1-\rho}{2}\|\mathbf{a}\|_2^2)
\end{equation}

The generalized Spiking LCA framework is able to incorporate the Elastic-Net formulation 
by modifying the slope of the activation function $T_{\lambda}(\cdot)$ as follows:

\begin{equation}
\label{elastic_net_lambda}
T_\lambda(a) \stackrel{\text { def }}{=} \begin{cases}0 & \text { if } a \leq \lambda \rho \\ \frac{a-\lambda \rho}{\lambda (1-\rho)+1} & \text { if } a>\lambda \rho\end{cases}
\end{equation}

To demonstrate the effectiveness of the generalized Spiking LCA on solving the Elastic-Net problem, we generate a dataset with sample size smaller than the total number of features, as an underdetermined problem. The target variable $\mathbf{s}$ is formed by combining sinusoidal signals with different frequencies. From the 100 frequencies in $\Phi$, only the lowest 10 are utilized to generate $\mathbf{s}$. The remaining features remain inert, rendering the feature space both high-dimensional and sparse, thus requiring a certain level of $l_1$-penalization. 

We then split the data into training and testing sets. In Fig.~\ref{fig:Elastic_Net}(a), the blue line represents the signal we aim to reconstruct, and the light red and orange dots are the sampled points. Due to noise, these sampled points deviate from the true signal values. The light red dots serve as the training dataset, used to determine the coefficient value $\mathbf{a}$, while the orange dots act as the testing dataset, to which we subsequently apply this determined value. The performance of algorithm is evaluated based on their goodness of fit score.  Fig.~\ref{fig:Elastic_Net}(b-d) display the results of the FISTA and the generalized Spiking LCA when applied to LASSO and Elastic-Net models. Both algorithms are run for 2 seconds. The generalized Spiking LCA consistently demonstrated more accurate predictions compared to FISTA. Moreover, the results underscore that  Elastic-Net outperforms in terms of $R^2$ score. While LASSO is renowned for its capability in sparse data recovery, it underperforms when the features are highly correlated. Indeed, when several correlated features influence the target, LASSO selects only one representative feature from the group and discards the others. This can lead to potential loss of information.  In contrast, Elastic-Net promotes sparsity in coefficient selection and slightly shrinks towards zero. Therefore, Elastic-Net adjusts their weights without eliminating them. This produces a less sparse model than a pure LASSO model. 

\subsubsection{LASSO}
Next, we extend the generalized Spiking LCA to solve the unconstrained LASSO problem. We introduce non-negative variables $\mathbf{a}^{+} \geq 0$ and $\mathbf{a}^{-} \geq 0$, such that $\mathbf{a}=\mathbf{a}^{+}-\mathbf{a}^{-}$ and $|\mathbf{a}|=\mathbf{a}^{+}+\mathbf{a}^{-}$. By defining $\widetilde{\Phi}=\left[\Phi , -\Phi\right] $ and $\mathbf{z}=\left[\mathbf{a}^{+} , \mathbf{a}^{-}\right]$, we reformulate the objective function of the unconstrained LASSO as:

\begin{equation}
\begin{gathered}
\min _{\mathbf{a} \in \mathbb{R}^n} \sum_{i=1}^n \mathbf{a}^{+}+\mathbf{a}^{-} \\
\text {s.t. } \Phi \mathbf{a}^{+}-\Phi \mathbf{a}^{-}
= \mathbf{s} . \\
\mathbf{a}^{+},\mathbf{a}^{-} \geq 0
\end{gathered} \quad  \equiv \begin{aligned}
& \min _{\mathbf{z} \ge 0}\|\mathbf{z}\|_1, \\
& \text { s.t. } \widetilde{\Phi} \mathbf{z}=\mathbf{s} .
\end
{aligned}
\end{equation}
The problem can be equivalently described using a Lagrangian multiplier approach as follows:

\begin{equation}
\label{LASSO_original}
\min _\mathbf{z} \frac{1}{2}\|\mathbf{s}-\widetilde{\Phi} \mathbf{z}\|_2^2+\lambda \sum_{k=1}^{2 N} z_k, \quad z_k \geq 0.
\end{equation}
Then we introduce a logarithmic barrier term to ensure that $z_k \geq 0$, transforming the constrained optimization problem into an unconstrained problem \cite{charles2011analog}. This can be written as:

\begin{equation}
\label{log_barrier}
\begin{aligned}
\min _\mathbf{z} & \frac{1}{2}\|\mathbf{x}-\widetilde{\Phi} \boldsymbol{z}\|_2^2+\lambda \sum_{k=1}^{2 N} z_k-\left(\frac{1}{\gamma}\right) \sum_{k=1}^{2 N} \log \left(z_k\right) \\
\text{s.t.} & \quad C\left(z_k\right)=z_k-\frac{\log \left(z_k\right)}{\gamma \lambda}
\end{aligned}
\end{equation}
where $C(z_k)$ represents a differentiable penalty function. The interconnection between $C(z_k)$ and $T_\lambda^{-1}(z_k)$ is given $\lambda \frac{d C(z_k)}{d z_k} = T_\lambda^{-1}(z_k) - z_k$, which allows one to solve for $z_k$ in terms of $u_k$ to determine the corresponding activation function $T_\lambda(\cdot)$. As an illustrative example, the generalized Spiking LCA has been effectively applied to reconstruct the CT image in Sec.~\ref{sec:CT}, which belongs to the unconstrained LASSO problem.

\subsubsection{Non-convex penalty functions}
\begin{figure*}[!t]
\centering
\subfloat[]{\includegraphics[height=1.75in]{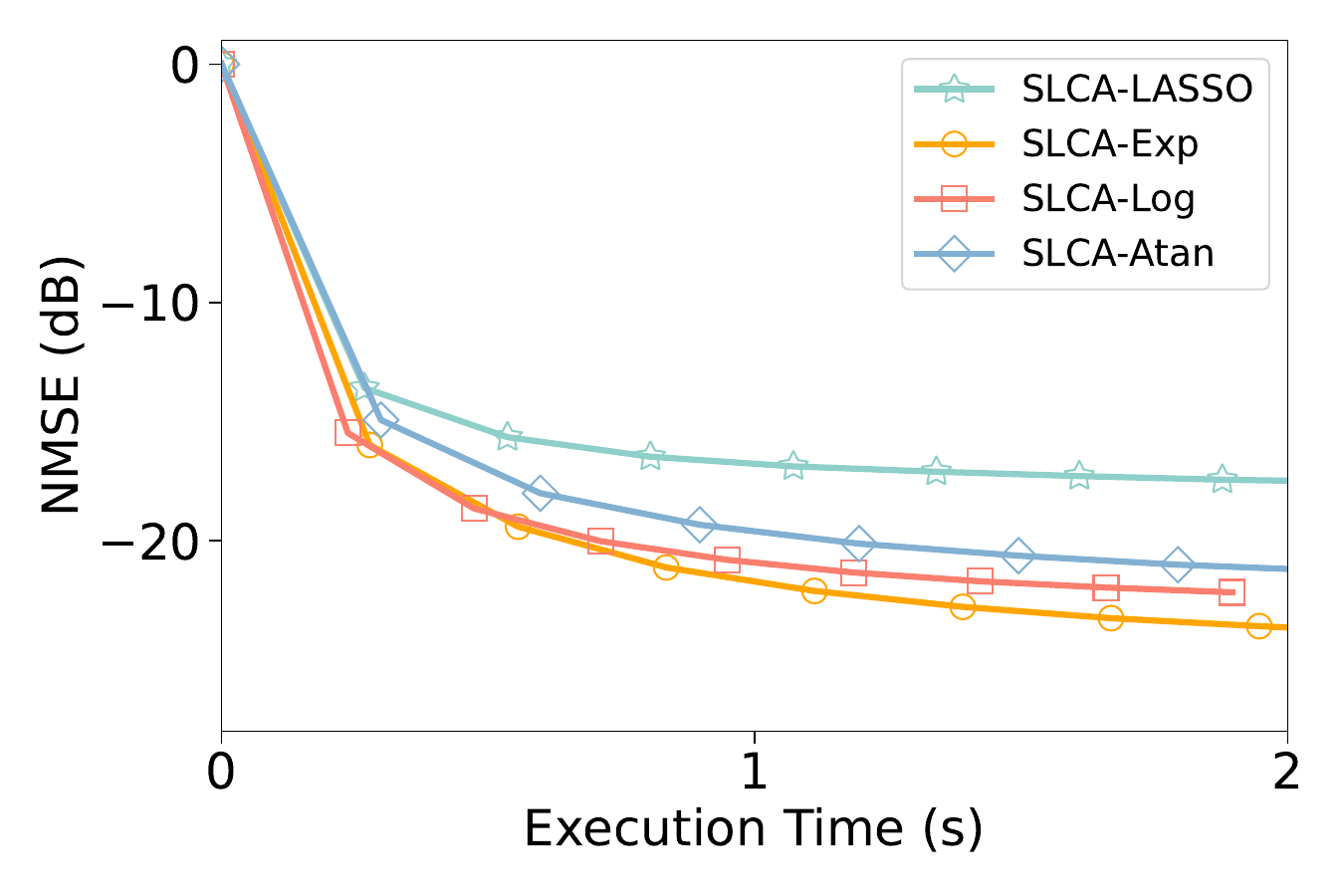}}
\hfil
\subfloat[]{\includegraphics[height=1.75in]{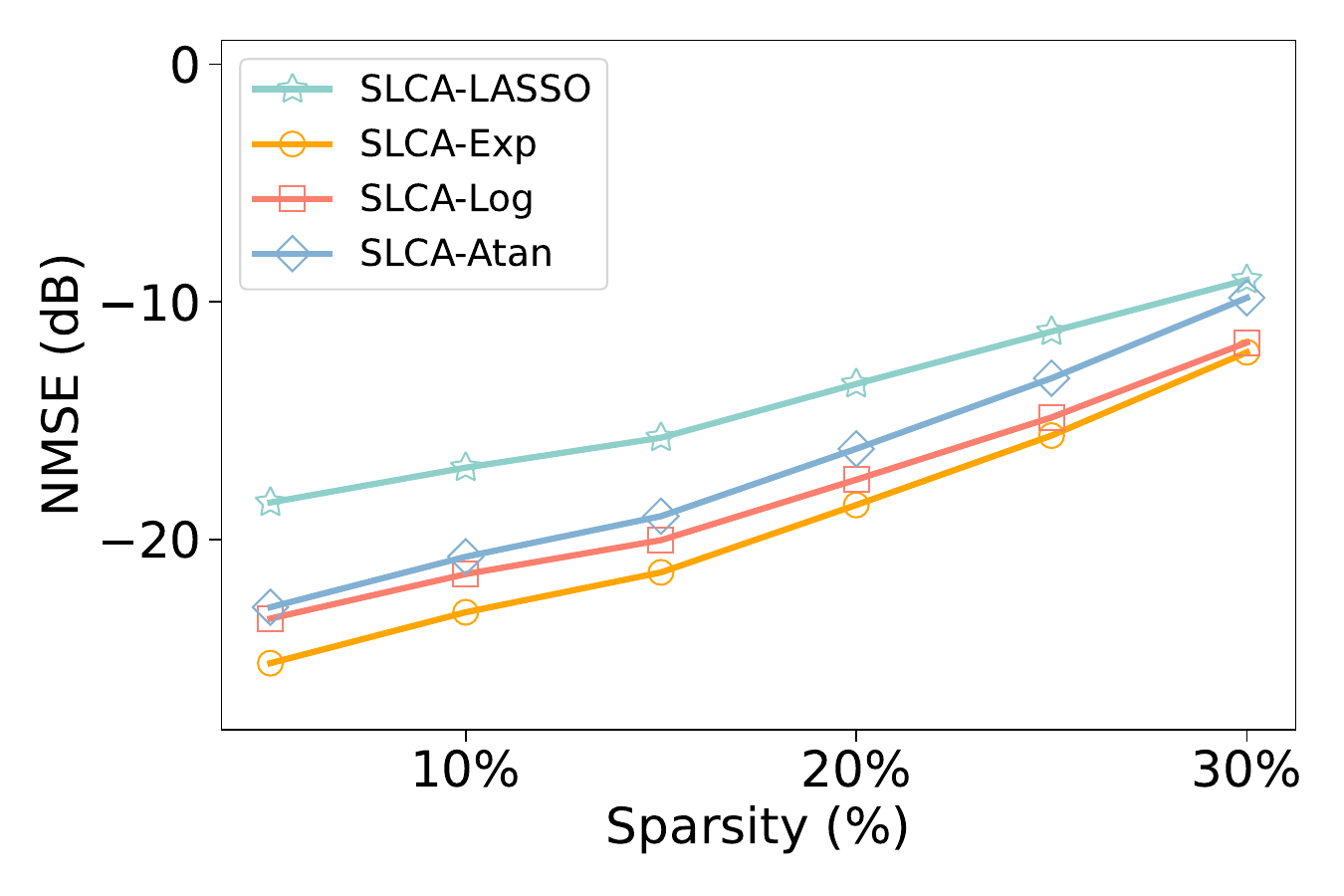}}
\caption{Comparison of the generalized Spiking LCA algorithm's performance using non-convex penalty functions. (a) Convergence among four different non-convex penalty functions. (b) Performance across diverse sparsity levels.}
\label{fig:multi_task}
\end{figure*}

In the above, we mainly focuse on convex penalty functions, including L1 and L2 regularity. Next we investigate the performance of generalized Spiking LCA with non-convex penalty functions, for instance, exponential, logarithmic, and arctangent functions. The exponential penalty is defined as $\widetilde{C}
(x)=1-e^{-\gamma x}$ with $\gamma > 0$. The logarithmic penalty is given by  $\widetilde{C}(x)=log(x + \theta)$ where $\theta \ge 1$, and the arctangent penalty is expressed as  $\widetilde{C}(x)= arctan(x/\eta)$ with $\eta > 0$. These non-convex functions need to satisfy rules 1-3 to ensure stability and convergence during experiments. It is obvious that these functions meet rules 1 and 2. For rule 3, there is a wide range of possible penalties that can be applied under an appropriate $\lambda$,  To meet this rule, the condition $\widetilde{C}^{\prime \prime}(x)>$ $-\frac{1}{\lambda}$ must be satisfied for all $x \in(0,+\infty)$.

In our numerical experiments, we set parameters $\gamma = 1$, $\theta = 1$, $\eta = 1$, and $\lambda = 0.1$. It is important to note that for vector $\mathbf{a}$, the function $\widetilde{C}(\cdot)$ operates on each component of the vector individually. For example, with $\widetilde{C}(x)=1-e^{-x}$, the function applied to vector $\mathbf{a}$ would be $\widetilde{C}(\mathbf{a})=1-e^{-\mathbf{a}}$, where the $i$-th component is $\widetilde{C}\left(a_i\right)=1-e^{-a_i}$. Additionally, the strategy applied to the current model adheres to the second approach delineated in Theorem 2, allowing for a robust and effective application of the theoretical framework. In our experiments, we focus on the signal recovery problem, same task described in Sec.~\ref{sec:CT}. The choice of penalty function also influences the algorithm's performance. The exponential penalty is the optimal choice among the penalties tested, as illustrated in Fig.~\ref{fig:multi_task}(a). Fig.~\ref{fig:multi_task}(b) illustrates the performance across varying sparsity levels, suggesting that the reconstruction quality improves as the signal becomes more sparse, which is evident from the decreasing NMSE values.

\section{Conclusion}

This paper develops a novel optimization algorithm for constructing spiking neural networks, generalized from the classical Spiking LCA. Our algorithm's unique feature is its flexibility, allowing it to handle various optimization problems within a unified network architecture. This is achieved by simply adjusting the external input current to neurons, which enhances the algorithm's versatility and minimizes the need for architectural adjustments, particularly when implemented on neuromorphic chips. The proposed spiking neural network is based on a large number of biologically plausible neuron models, encompassing a range from the LIF model to the Hodgkin-Huxley type model. Besides, we have demonstrated our model's practical utility and efficiency by applying it to various real-world problems related to compressed sensing and signal processing. By outperforming popular optimization algorithms, such as FISTA, our model confirms its capability to address large-scale optimization problems more effectively.

\section*{Acknowledgments}
\noindent This work was supported by Science and Technology Innovation 2030 - Brain Science and Brain-Inspired Intelligence Project with Grant No. 2021ZD0200204  (S.L., D.Z.); National Natural Science Foundation of China Grant 12271361, 12250710674 (S.L.); National Natural Science Foundation of China with Grant No.  12225109, 12071287 (D.Z.), Lingang Laboratory Grant No. LG-QS-202202-01, and the Student Innovation Center at Shanghai Jiao Tong University (X.-x.D., Z.-q.K.T., S.L., and D.Z.).

\section*{CRediT authorship contribution statement}

\textbf{Xuexing Du:} Conceptualization, Methodology, Software, Visualization, Writing – original draft. \textbf{Zhong-qi K. Tian:} Conceptualization. \textbf{Songting Li:} Conceptualization, Methodology, Funding acquisition, Supervision, Software, Visualization, Writing – original draft, Writing – review \& editing. \textbf{Douglas Zhou:} Conceptualization, Methodology, Funding acquisition, Supervision, Software, Visualization, Writing – original draft, Writing – review \& editing.

\section*{Declaration of competing interest}
The authors declare that they have no known competing financial interests or personal relationships that could have appeared to
influence the work reported in this paper.

\appendix
\section{Appendix}
\subsection{Neuron Models}
\label{appendix:neuron_model}
\noindent \textbf{GIF model.}  The dynamics of the $i^{th}$ neuron of a generalized leaky integrate-and-fire ($\mathrm{GIF}$ ) network is governed~\cite{mihalacs2009generalized, teeter2018generalized}
\begin{equation}
\begin{aligned}
c \frac{\mathrm{d} v_i}{\mathrm{~d} t} & = -g_L\left(v_i-v_{L}\right)+ \sum_j I_j+ I_i^{\text {input }}, \\
\frac{\mathrm{d} \theta_i}{\mathrm{d} t} & =a\left(v_i- v_{L}\right)-b\left(\theta_i-\theta_{\infty}\right), \\[10pt]
\frac{\mathrm{d} I_j}{\mathrm{~d} t} & =-k_j I_j, \quad j=1,2, \ldots, n \\[10pt]
\text { if } v_i & >\theta_i, \quad I_j \leftarrow r_j I_j+A_j, v_i \leftarrow v_{\text {reset }}, \theta_i \leftarrow \max \left(\theta_{\text{reset}}, \theta_i\right),
\end{aligned}
\end{equation}
where $v_i$ represents the membrane potential of the neuron, and $c$ refer to the membrane capacitance. The terms $g_L$ and $v_L$ are used to denote the leak conductance and the reversal potential, respectively. $I_j$ represents an arbitrary number of internal currents. These currents are primarily influenced by the dynamics of ion channels, providing the model with the flexibility to simulate various neuronal firing patterns. Additionally, $I_i^{\text{input}}$ denotes the externally injected current. The model encompasses a total of $n+2$ variables: the membrane potential $v_i$, the membrane potential threshold $\theta$, and $n$ internal currents. The decay of each internal current $I_j$ is described by a third differential equation, with a decay rate of $k_j$.

In its formulation, the GIF model extends the Leaky Integrate-and-Fire (LIF) model by integrating the effects of internal ionic currents into the first differential equation. The second differential equation elaborates on the dynamics of the firing threshold $\theta$, which is not constant. The first term of this equation accounts for the influence of the membrane potential on $\theta$, while the second term delineates the decay of $\theta$ towards its equilibrium value $\theta_{\infty}$ at a decay rate $\mathrm{b}$. Upon the firing of a neuron, the membrane potential $v_i$ and the threshold $\theta$ are reset, and the internal currents are adjusted according to specific rules. 

In the numerical simulation, we set parameters as $c = 1 \mathrm{~\mu F}\cdot \mathrm{~cm}^{-2}$, $g_L =0.05 \mathrm{~mS}\cdot\mathrm{~cm}^{-2}$,    $v_{L}= -70 \mathrm{~mV}$, $v_{reset}= -70 \mathrm{~mV}$, $\theta_{\infty}= -50 \mathrm{~mV}$, $\theta_{reset}= -60 \mathrm{~mV}$, $n=0$,  $r = 20$, $a = 0$, $b = 0.01$, $k_1= 0.2$, $k_2 = 0.02$. 

\noindent\textbf{Morris-Lecar Model.} The dynamics of the $i^{th}$ neuron of a Morris-Lecar network is governed by \cite{tsumoto2006bifurcations}
\begin{equation}
\begin{aligned}
 c \frac{d v_i}{d t}= & -g_{{\mathrm{Ca}}} m_i^{\infty}\left(v_i-v_{{\mathrm{Ca}}}\right)-g_K w_i\left(v_i-v_{\mathrm{K}}\right)\\
& -g_{\mathrm{L}}\left(v_i-v_{\mathrm{L}}\right)+I_i^{\text {input }} \\
\frac{d w_i}{d t}= & \frac{1}{\tau_i^w}\left(w_i^
{\infty}-w_i\right),
\end{aligned}
\end{equation}
where $c$ and $v_i$ denote the neuron's membrane capacitance and membrane potential, respectively.  $I_i^{\text {input }}$  is the injected current. The terms $v_{\mathrm{Ca}}, v_{\mathrm{K}}$, and $v_{\mathrm{L}}$ represent the reversal potentials for the calcium, potassium, and leak currents, respectively. Correspondingly, $g_{\mathrm{Ca}}, g_{\mathrm{K}}$, and $g_{\mathrm{L}}$ are the maximum conductances for these currents. The variable $w_i$ refers to the neuron's recovery variable, which is a normalized potassium conductance. The parameters $m_i^{\infty}$ and $w_i^{\infty}$ are the voltage-dependent equilibrium values for the normalized conductances of calcium and potassium, respectively, and are defined by
\begin{equation}
\begin{aligned}
& m_i^{\infty}=0.5\left(1+\tanh \left[\left(v_i-V_1\right) / V_2\right]\right) \\
& w_i^{\infty}=0.5\left(1+\tanh \left[\left(v_i-V_3\right) / V_4\right]\right),
\end{aligned}
\end{equation}
where $V_1, V_2, V_3$, and $V_4$ are the constant parameters. $\tau_i^w$ is a voltage-dependent time constant of $w_i$, defined by
\begin{equation}
\tau_i^w=\phi\left(\cosh \frac{v_i-V_3}{2 V_4}\right)^{-1}
\end{equation}
where $\phi$ is a temperature-dependent parameter, fixed as a constant in simulation. When the voltage reaches threshold $v_{th}$, the neuron emits a spike to all its postsynaptic neurons. The parameters are set in numerical simulations as 
$g_{\mathrm{Ca}}=4.4 \mathrm{~mS} \cdot \mathrm{~cm}^{-2}$, $v_{\mathrm{Ca}}=130 \mathrm{~mV}$, $g_{\mathrm{K}}=8 \mathrm{~mS} \cdot \mathrm{cm}^{-2}$, $v_{\mathrm{K}}=-84 \mathrm{~mV}$, $g_{\mathrm{L}}=2 \mathrm{~mS} \cdot \mathrm{cm}^{-2}$, $v_{\mathrm{L}}=-60 \mathrm{~mV}$,  $c=20 \mathrm{~\mu F} \cdot \mathrm{cm}^{-2}$, $V_1=-1.2 \mathrm{~mV}$, $V_2=18 \mathrm{~mV}$, $V_3=2 \mathrm{~mV}$, $V_4= 30 \mathrm{~mV}$, $\phi=0.04$, and $v_{\text {th}}=0 \mathrm{~mV}$.

\noindent \textbf{Wang-Buzsaki model.} The dynamics of the $i^{th}$ neuron of a Wang-Buzsaki network is governed by \cite{wang1996gamma}

\begin{equation}
c \frac{d v_i}{d t}=-\left(I_{N a}+I_K+I_L\right)+I_i^{\text {input }},
\end{equation}
where $c$ is the cell membrane capacitance; $v_i$ is the membrane potential (voltage);  $I_i^{\text {input }}$  is the injected current. The leak current $I_L= g_{\mathrm{L}}\left(v_i-v_{\mathrm{L}}\right)$,  and the transient sodium current
\begin{equation}
I_{\mathrm{Na}}=g_{\mathrm{Na}} m_{i,\infty}^3 h\left(v_i-v_{\mathrm{Na}}\right),
\end{equation}
where the activation variable $m_i$ is assumed fast and substituted by its steady-state function $m_{i,\infty}=\alpha_m/\left(\alpha_m+\beta_m\right)$. Additionally, the inactivation variable $h_i$ follows first-order kinetics:
\begin{equation}
\frac{d h_i}{d t}=\phi\left(\alpha_h(1-h_i)-\beta_h h_i\right). 
\end{equation}
The delayed rectifier potassium current
\begin{equation}
    I_\mathrm{K}=g_\mathrm{k} n_i^4\left(v_i-v_\mathrm{K}\right),
\end{equation}
where the activation variable $n$ obeys the following equation:
 \begin{equation}
\frac{d n_i}{d t}=\phi\left(\alpha_n(1-n_i)-\beta_n n_i\right).
\end{equation}
The $m_i, h_i$, and $n_i$ are gating variables; $v_{\mathrm{Na}}, v_{\mathrm{K}}$, and $v_L$ are the reversal potentials for the sodium, potassium, and leak currents, respectively. Meanwhile, $g_{\mathrm{Na}}$, $g_{\mathrm{K}}$, and $g_L$ correspond to the maximum conductances for these currents.The constant $\phi$ serves as a temperature regulation factor. The rate variables $\alpha_z$ and $\beta_z$ ($z = m,h,n$) are defined as:
\begin{equation}
\begin{array}{ll}
 \alpha_m(v)=-\dfrac{0.1(v+35)}{\exp (-0.1(v+35))-1}, & \beta_m(v)=4 \exp (-\dfrac{v+60}{18}),\\[10pt]
\alpha_h(v)=0.07 \exp (-\dfrac{v+58}{20}), & \beta_h(v)=1 /(\exp (-0.1(v+28))+1), \\[10pt]
\alpha_n(v)= \dfrac{-0.01(v+34)}{\exp (-0.1(v+34))-1}, & \beta_n(v)=0.125 \exp (-\dfrac{v+44}{80}) .
\end{array}
\end{equation}

We take the parameters as in Ref. \cite{wang1996gamma} that $c=1  \mathrm{~\mu F} \cdot \mathrm{~cm}^{-2}$, $v_{\mathrm{Na}}=55 \mathrm{~mV}$, $v_{\mathrm{K}}=-90 \mathrm{~mV}$, $v_L=-65 \mathrm{~mV}$, $g_{\mathrm{Na}}= 35 \mathrm{~mS} \cdot \mathrm{cm}^{-2}$, $g_{\mathrm{K}}= 9 \mathrm{~mS} \cdot \mathrm{cm}^{-2}$,  $g_L=0.1 \mathrm{~mS} \cdot \mathrm{cm}^{-2}$ and $\phi =5$. When the voltage $v_i$ reaches the firing threshold, $v_{\mathrm{th}}= 20 \mathrm{~mV}$, we say the $i^{th}$ neuron generates a spike at this time.




\begin{thebibliography}{55}
\expandafter\ifx\csname natexlab\endcsname\relax\def\natexlab#1{#1}\fi
\providecommand{\url}[1]{\texttt{#1}}
\providecommand{\href}[2]{#2}
\providecommand{\path}[1]{#1}
\providecommand{\DOIprefix}{doi:}
\providecommand{\ArXivprefix}{arXiv:}
\providecommand{\URLprefix}{URL: }
\providecommand{\Pubmedprefix}{pmid:}
\providecommand{\doi}[1]{\href{http://dx.doi.org/#1}{\path{#1}}}
\providecommand{\Pubmed}[1]{\href{pmid:#1}{\path{#1}}}
\providecommand{\bibinfo}[2]{#2}
\ifx\xfnm\relax \def\xfnm[#1]{\unskip,\space#1}\fi
\bibitem[{Aitchison and Lengyel(2017)}]{aitchison2017or}
\bibinfo{author}{Aitchison, L.}, \bibinfo{author}{Lengyel, M.},
  \bibinfo{year}{2017}.
\newblock \bibinfo{title}{With or without you: predictive coding and {Bayesian}
  inference in the brain}.
\newblock \bibinfo{journal}{Current opinion in neurobiology}
  \bibinfo{volume}{46}, \bibinfo{pages}{219--227}.
\bibitem[{Balavoine et~al.(2012)Balavoine, Romberg and
  Rozell}]{balavoine2012convergence}
\bibinfo{author}{Balavoine, A.}, \bibinfo{author}{Romberg, J.},
  \bibinfo{author}{Rozell, C.J.}, \bibinfo{year}{2012}.
\newblock \bibinfo{title}{Convergence and rate analysis of neural networks for
  sparse approximation}.
\newblock \bibinfo{journal}{IEEE transactions on neural networks and learning
  systems} \bibinfo{volume}{23}, \bibinfo{pages}{1377--1389}.
\bibitem[{Balavoine et~al.(2013)Balavoine, Rozell and
  Romberg}]{balavoine2013convergence}
\bibinfo{author}{Balavoine, A.}, \bibinfo{author}{Rozell, C.J.},
  \bibinfo{author}{Romberg, J.}, \bibinfo{year}{2013}.
\newblock \bibinfo{title}{Convergence of a neural network for sparse
  approximation using the nonsmooth {$\mathrm{\L}$}ojasiewicz inequality}, in:
  \bibinfo{booktitle}{The 2013 International Joint Conference on Neural
  Networks (IJCNN)}, \bibinfo{organization}{IEEE}. pp. \bibinfo{pages}{1--8}.
\bibitem[{Barranca(2021)}]{barranca2021neural}
\bibinfo{author}{Barranca, V.J.}, \bibinfo{year}{2021}.
\newblock \bibinfo{title}{Neural network learning of improved compressive
  sensing sampling and receptive field structure}.
\newblock \bibinfo{journal}{Neurocomputing} \bibinfo{volume}{455},
  \bibinfo{pages}{368--378}.
\bibitem[{Barranca et~al.(2019)Barranca, Kova{\v{c}}i{\v{c}} and
  Zhou}]{barranca2019role}
\bibinfo{author}{Barranca, V.J.}, \bibinfo{author}{Kova{\v{c}}i{\v{c}}, G.},
  \bibinfo{author}{Zhou, D.}, \bibinfo{year}{2019}.
\newblock \bibinfo{title}{The role of sparsity in inverse problems for networks
  with nonlinear dynamics}.
\newblock \bibinfo{journal}{Communications in Mathematical Sciences}
  \bibinfo{volume}{17}.
\bibitem[{Barranca et~al.(2014)Barranca, Kova{\v{c}}i{\v{c}}, Zhou and
  Cai}]{barranca2014sparsity}
\bibinfo{author}{Barranca, V.J.}, \bibinfo{author}{Kova{\v{c}}i{\v{c}}, G.},
  \bibinfo{author}{Zhou, D.}, \bibinfo{author}{Cai, D.}, \bibinfo{year}{2014}.
\newblock \bibinfo{title}{Sparsity and compressed coding in sensory systems}.
\newblock \bibinfo{journal}{PLoS computational biology} \bibinfo{volume}{10},
  \bibinfo{pages}{e1003793}.
\bibitem[{Barranca and Zhou(2019)}]{barranca2019compressive}
\bibinfo{author}{Barranca, V.J.}, \bibinfo{author}{Zhou, D.},
  \bibinfo{year}{2019}.
\newblock \bibinfo{title}{Compressive sensing inference of neuronal network
  connectivity in balanced neuronal dynamics}.
\newblock \bibinfo{journal}{Frontiers in Neuroscience} \bibinfo{volume}{13},
  \bibinfo{pages}{492216}.
\bibitem[{Beck and Teboulle(2009)}]{beck2009fast}
\bibinfo{author}{Beck, A.}, \bibinfo{author}{Teboulle, M.},
  \bibinfo{year}{2009}.
\newblock \bibinfo{title}{A fast iterative shrinkage-thresholding algorithm for
  linear inverse problems}.
\newblock \bibinfo{journal}{SIAM journal on imaging sciences}
  \bibinfo{volume}{2}, \bibinfo{pages}{183--202}.
\bibitem[{Boutin et~al.(2021)Boutin, Franciosini, Chavane, Ruffier and
  Perrinet}]{boutin2021sparse}
\bibinfo{author}{Boutin, V.}, \bibinfo{author}{Franciosini, A.},
  \bibinfo{author}{Chavane, F.}, \bibinfo{author}{Ruffier, F.},
  \bibinfo{author}{Perrinet, L.}, \bibinfo{year}{2021}.
\newblock \bibinfo{title}{Sparse deep predictive coding captures contour
  integration capabilities of the early visual system}.
\newblock \bibinfo{journal}{PLoS computational biology} \bibinfo{volume}{17},
  \bibinfo{pages}{e1008629}.
\bibitem[{Chai et~al.(2022)Chai, Fu, Gan, Lu and Zhang}]{chai2022image}
\bibinfo{author}{Chai, X.}, \bibinfo{author}{Fu, J.}, \bibinfo{author}{Gan,
  Z.}, \bibinfo{author}{Lu, Y.}, \bibinfo{author}{Zhang, Y.},
  \bibinfo{year}{2022}.
\newblock \bibinfo{title}{An image encryption scheme based on multi-objective
  optimization and block compressed sensing}.
\newblock \bibinfo{journal}{Nonlinear Dynamics} \bibinfo{volume}{108},
  \bibinfo{pages}{2671--2704}.
\bibitem[{Charles et~al.(2011)Charles, Garrigues and
  Rozell}]{charles2011analog}
\bibinfo{author}{Charles, A.S.}, \bibinfo{author}{Garrigues, P.},
  \bibinfo{author}{Rozell, C.J.}, \bibinfo{year}{2011}.
\newblock \bibinfo{title}{Analog sparse approximation with applications to
  compressed sensing}.
\newblock \bibinfo{journal}{arXiv preprint arXiv:1111.4118} .
\bibitem[{Chavez~Arana et~al.(2022)Chavez~Arana, Renner and
  Sornborger}]{chavez2022neuromorphic}
\bibinfo{author}{Chavez~Arana, D.}, \bibinfo{author}{Renner, A.},
  \bibinfo{author}{Sornborger, A.}, \bibinfo{year}{2022}.
\newblock \bibinfo{title}{A neuromorphic normalization algorithm for
  stabilizing synaptic weights with application to dictionary learning in
  {LCA}}, in: \bibinfo{booktitle}{Proceedings of the 2022 Annual Neuro-Inspired
  Computational Elements Conference}, pp. \bibinfo{pages}{58--60}.
\bibitem[{Chavez~Arana et~al.(2023)Chavez~Arana, Renner and
  Sornborger}]{chavez2023spiking}
\bibinfo{author}{Chavez~Arana, D.}, \bibinfo{author}{Renner, A.},
  \bibinfo{author}{Sornborger, A.}, \bibinfo{year}{2023}.
\newblock \bibinfo{title}{Spiking {LCA} in a neural circuit with dictionary
  learning and synaptic normalization}, in: \bibinfo{booktitle}{Proceedings of
  the 2023 Annual Neuro-Inspired Computational Elements Conference}, pp.
  \bibinfo{pages}{47--51}.
\bibitem[{Chen and Gu(2014)}]{chen2014convergence}
\bibinfo{author}{Chen, L.}, \bibinfo{author}{Gu, Y.}, \bibinfo{year}{2014}.
\newblock \bibinfo{title}{The convergence guarantees of a non-convex approach
  for sparse recovery}.
\newblock \bibinfo{journal}{IEEE Transactions on Signal Processing}
  \bibinfo{volume}{62}, \bibinfo{pages}{3754--3767}.
\bibitem[{Chou et~al.(2018)Chou, Chung and Lu}]{chou2018algorithmic}
\bibinfo{author}{Chou, C.N.}, \bibinfo{author}{Chung, K.M.},
  \bibinfo{author}{Lu, C.J.}, \bibinfo{year}{2018}.
\newblock \bibinfo{title}{On the algorithmic power of spiking neural networks}.
\newblock \bibinfo{journal}{arXiv preprint arXiv:1803.10375} .
\bibitem[{Davies et~al.(2018)Davies, Srinivasa, Lin, Chinya, Cao, Choday,
  Dimou, Joshi, Imam, Jain et~al.}]{davies2018loihi}
\bibinfo{author}{Davies, M.}, \bibinfo{author}{Srinivasa, N.},
  \bibinfo{author}{Lin, T.H.}, \bibinfo{author}{Chinya, G.},
  \bibinfo{author}{Cao, Y.}, \bibinfo{author}{Choday, S.H.},
  \bibinfo{author}{Dimou, G.}, \bibinfo{author}{Joshi, P.},
  \bibinfo{author}{Imam, N.}, \bibinfo{author}{Jain, S.}, et~al.,
  \bibinfo{year}{2018}.
\newblock \bibinfo{title}{Loihi: A neuromorphic manycore processor with on-chip
  learning}.
\newblock \bibinfo{journal}{Ieee Micro} \bibinfo{volume}{38},
  \bibinfo{pages}{82--99}.
\bibitem[{Davies et~al.(2021)Davies, Wild, Orchard, Sandamirskaya, Guerra,
  Joshi, Plank and Risbud}]{davies2021advancing}
\bibinfo{author}{Davies, M.}, \bibinfo{author}{Wild, A.},
  \bibinfo{author}{Orchard, G.}, \bibinfo{author}{Sandamirskaya, Y.},
  \bibinfo{author}{Guerra, G.A.F.}, \bibinfo{author}{Joshi, P.},
  \bibinfo{author}{Plank, P.}, \bibinfo{author}{Risbud, S.R.},
  \bibinfo{year}{2021}.
\newblock \bibinfo{title}{Advancing neuromorphic computing with loihi: A survey
  of results and outlook}.
\newblock \bibinfo{journal}{Proceedings of the IEEE} \bibinfo{volume}{109},
  \bibinfo{pages}{911--934}.
\bibitem[{De~Maio et~al.(2019)De~Maio, Eldar and Haimovich}]{de2019compressed}
\bibinfo{author}{De~Maio, A.}, \bibinfo{author}{Eldar, Y.C.},
  \bibinfo{author}{Haimovich, A.M.}, \bibinfo{year}{2019}.
\newblock \bibinfo{title}{Compressed sensing in radar signal processing}.
\newblock \bibinfo{publisher}{Cambridge University Press}.
\bibitem[{Donoho(2006)}]{donoho2006compressed}
\bibinfo{author}{Donoho, D.L.}, \bibinfo{year}{2006}.
\newblock \bibinfo{title}{Compressed sensing}.
\newblock \bibinfo{journal}{IEEE Transactions on information theory}
  \bibinfo{volume}{52}, \bibinfo{pages}{1289--1306}.
\bibitem[{Fair et~al.(2019)Fair, Mendat, Andreou, Rozell, Romberg and
  Anderson}]{fair2019sparse}
\bibinfo{author}{Fair, K.L.}, \bibinfo{author}{Mendat, D.R.},
  \bibinfo{author}{Andreou, A.G.}, \bibinfo{author}{Rozell, C.J.},
  \bibinfo{author}{Romberg, J.}, \bibinfo{author}{Anderson, D.V.},
  \bibinfo{year}{2019}.
\newblock \bibinfo{title}{Sparse coding using the locally competitive algorithm
  on the {TrueNorth} neurosynaptic system}.
\newblock \bibinfo{journal}{Frontiers in neuroscience} \bibinfo{volume}{13},
  \bibinfo{pages}{754}.
\bibitem[{Field(1994)}]{field1994goal}
\bibinfo{author}{Field, D.J.}, \bibinfo{year}{1994}.
\newblock \bibinfo{title}{What is the goal of sensory coding?}
\newblock \bibinfo{journal}{Neural computation} \bibinfo{volume}{6},
  \bibinfo{pages}{559--601}.
\bibitem[{Fosson(2018)}]{fosson2018biconvex}
\bibinfo{author}{Fosson, S.M.}, \bibinfo{year}{2018}.
\newblock \bibinfo{title}{A biconvex analysis for lasso $\ell_1$ reweighting}.
\newblock \bibinfo{journal}{IEEE Signal Processing Letters}
  \bibinfo{volume}{25}, \bibinfo{pages}{1795--1799}.
\bibitem[{Gregor and LeCun(2010)}]{gregor2010learning}
\bibinfo{author}{Gregor, K.}, \bibinfo{author}{LeCun, Y.},
  \bibinfo{year}{2010}.
\newblock \bibinfo{title}{Learning fast approximations of sparse coding}.
\newblock \bibinfo{journal}{Proceedings of the 27th international conference on
  international conference on machine learning} , \bibinfo{pages}{399--406}.
\bibitem[{He et~al.(2022)He, Chen, Zhong and Wu}]{he2022granular}
\bibinfo{author}{He, L.}, \bibinfo{author}{Chen, Y.}, \bibinfo{author}{Zhong,
  C.}, \bibinfo{author}{Wu, K.}, \bibinfo{year}{2022}.
\newblock \bibinfo{title}{Granular elastic network regression with stochastic
  gradient descent}.
\newblock \bibinfo{journal}{Mathematics} \bibinfo{volume}{10},
  \bibinfo{pages}{2628}.
\bibitem[{Henke et~al.(2022)Henke, Teti, Kenyon, Migliori and
  Kunde}]{henke2022apples}
\bibinfo{author}{Henke, K.}, \bibinfo{author}{Teti, M.},
  \bibinfo{author}{Kenyon, G.}, \bibinfo{author}{Migliori, B.},
  \bibinfo{author}{Kunde, G.}, \bibinfo{year}{2022}.
\newblock \bibinfo{title}{Apples-to-spikes: The first detailed comparison of
  {LASSO} solutions generated by a spiking neuromorphic processor}, in:
  \bibinfo{booktitle}{Proceedings of the International Conference on
  Neuromorphic Systems 2022}, pp. \bibinfo{pages}{1--8}.
\bibitem[{Kougioumtzoglou et~al.(2020)Kougioumtzoglou, Petromichelakis and
  Psaros}]{kougioumtzoglou2020sparse}
\bibinfo{author}{Kougioumtzoglou, I.A.}, \bibinfo{author}{Petromichelakis, I.},
  \bibinfo{author}{Psaros, A.F.}, \bibinfo{year}{2020}.
\newblock \bibinfo{title}{Sparse representations and compressive sampling
  approaches in engineering mechanics: A review of theoretical concepts and
  diverse applications}.
\newblock \bibinfo{journal}{Probabilistic Engineering Mechanics}
  \bibinfo{volume}{61}, \bibinfo{pages}{103082}.
\bibitem[{Kuzin et~al.(2019)Kuzin, Isupova and Mihaylova}]{kuzin2019bayesian}
\bibinfo{author}{Kuzin, D.}, \bibinfo{author}{Isupova, O.},
  \bibinfo{author}{Mihaylova, L.}, \bibinfo{year}{2019}.
\newblock \bibinfo{title}{Bayesian neural networks for sparse coding}.
\newblock \bibinfo{journal}{ICASSP 2019-2019 IEEE International Conference on
  Acoustics, Speech and Signal Processing (ICASSP)} ,
  \bibinfo{pages}{2992--2996}.
\bibitem[{Li et~al.(2021)Li, Wei and Zhang}]{li2021double}
\bibinfo{author}{Li, Y.M.}, \bibinfo{author}{Wei, D.}, \bibinfo{author}{Zhang,
  L.}, \bibinfo{year}{2021}.
\newblock \bibinfo{title}{Double-encrypted watermarking algorithm based on
  cosine transform and fractional fourier transform in invariant wavelet
  domain}.
\newblock \bibinfo{journal}{Information Sciences} \bibinfo{volume}{551},
  \bibinfo{pages}{205--227}.
\bibitem[{Mairal et~al.(2009)Mairal, Bach, Ponce and Sapiro}]{mairal2009online}
\bibinfo{author}{Mairal, J.}, \bibinfo{author}{Bach, F.},
  \bibinfo{author}{Ponce, J.}, \bibinfo{author}{Sapiro, G.},
  \bibinfo{year}{2009}.
\newblock \bibinfo{title}{Online dictionary learning for sparse coding}, in:
  \bibinfo{booktitle}{Proceedings of the 26th annual international conference
  on machine learning}, pp. \bibinfo{pages}{689--696}.
\bibitem[{Mihala{\c{s}} and Niebur(2009)}]{mihalacs2009generalized}
\bibinfo{author}{Mihala{\c{s}}, {\c{S}}.}, \bibinfo{author}{Niebur, E.},
  \bibinfo{year}{2009}.
\newblock \bibinfo{title}{A generalized linear integrate-and-fire neural model
  produces diverse spiking behaviors}.
\newblock \bibinfo{journal}{Neural computation} \bibinfo{volume}{21},
  \bibinfo{pages}{704--718}.
\bibitem[{Olshausen and Field(1996)}]{olshausen1996emergence}
\bibinfo{author}{Olshausen, B.A.}, \bibinfo{author}{Field, D.J.},
  \bibinfo{year}{1996}.
\newblock \bibinfo{title}{Emergence of simple-cell receptive field properties
  by learning a sparse code for natural images}.
\newblock \bibinfo{journal}{Nature} \bibinfo{volume}{381},
  \bibinfo{pages}{607--609}.
\bibitem[{Rehn and Sommer(2007)}]{rehn2007network}
\bibinfo{author}{Rehn, M.}, \bibinfo{author}{Sommer, F.T.},
  \bibinfo{year}{2007}.
\newblock \bibinfo{title}{A network that uses few active neurones to code
  visual input predicts the diverse shapes of cortical receptive fields}.
\newblock \bibinfo{journal}{Journal of computational neuroscience}
  \bibinfo{volume}{22}, \bibinfo{pages}{135--146}.
\bibitem[{Rozell et~al.(2008)Rozell, Johnson, Baraniuk and
  Olshausen}]{rozell2008sparse}
\bibinfo{author}{Rozell, C.J.}, \bibinfo{author}{Johnson, D.H.},
  \bibinfo{author}{Baraniuk, R.G.}, \bibinfo{author}{Olshausen, B.A.},
  \bibinfo{year}{2008}.
\newblock \bibinfo{title}{Sparse coding via thresholding and local competition
  in neural circuits}.
\newblock \bibinfo{journal}{Neural computation} \bibinfo{volume}{20},
  \bibinfo{pages}{2526--2563}.
\bibitem[{Sachdev et~al.(2012)Sachdev, Krause and Mazer}]{sachdev2012surround}
\bibinfo{author}{Sachdev, R.N.}, \bibinfo{author}{Krause, M.R.},
  \bibinfo{author}{Mazer, J.A.}, \bibinfo{year}{2012}.
\newblock \bibinfo{title}{Surround suppression and sparse coding in visual and
  barrel cortices}.
\newblock \bibinfo{journal}{Frontiers in neural circuits} \bibinfo{volume}{6},
  \bibinfo{pages}{43}.
\bibitem[{Shapero et~al.(2014)Shapero, Zhu, Hasler and
  Rozell}]{shapero2014optimal}
\bibinfo{author}{Shapero, S.}, \bibinfo{author}{Zhu, M.},
  \bibinfo{author}{Hasler, J.}, \bibinfo{author}{Rozell, C.},
  \bibinfo{year}{2014}.
\newblock \bibinfo{title}{Optimal sparse approximation with integrate and fire
  neurons}.
\newblock \bibinfo{journal}{International journal of neural systems}
  \bibinfo{volume}{24}, \bibinfo{pages}{1440001}.
\bibitem[{Shen and Strang(1998)}]{shen1998asymptotics}
\bibinfo{author}{Shen, J.}, \bibinfo{author}{Strang, G.}, \bibinfo{year}{1998}.
\newblock \bibinfo{title}{Asymptotics of daubechies filters, scaling functions,
  and wavelets}.
\newblock \bibinfo{journal}{Applied and Computational Harmonic Analysis}
  \bibinfo{volume}{5}, \bibinfo{pages}{312--331}.
\bibitem[{Tang et~al.(2017)Tang, Lin and Davies}]{tang2017sparse}
\bibinfo{author}{Tang, P.T.P.}, \bibinfo{author}{Lin, T.H.},
  \bibinfo{author}{Davies, M.}, \bibinfo{year}{2017}.
\newblock \bibinfo{title}{Sparse coding by spiking neural networks: Convergence
  theory and computational results}.
\newblock \bibinfo{journal}{arXiv preprint arXiv:1705.05475} .
\bibitem[{Teeter et~al.(2018)Teeter, Iyer, Menon, Gouwens, Feng, Berg, Szafer,
  Cain, Zeng, Hawrylycz et~al.}]{teeter2018generalized}
\bibinfo{author}{Teeter, C.}, \bibinfo{author}{Iyer, R.},
  \bibinfo{author}{Menon, V.}, \bibinfo{author}{Gouwens, N.},
  \bibinfo{author}{Feng, D.}, \bibinfo{author}{Berg, J.},
  \bibinfo{author}{Szafer, A.}, \bibinfo{author}{Cain, N.},
  \bibinfo{author}{Zeng, H.}, \bibinfo{author}{Hawrylycz, M.}, et~al.,
  \bibinfo{year}{2018}.
\newblock \bibinfo{title}{Generalized leaky integrate-and-fire models classify
  multiple neuron types}.
\newblock \bibinfo{journal}{Nature communications} \bibinfo{volume}{9},
  \bibinfo{pages}{709}.
\bibitem[{Tsumoto et~al.(2006)Tsumoto, Kitajima, Yoshinaga, Aihara and
  Kawakami}]{tsumoto2006bifurcations}
\bibinfo{author}{Tsumoto, K.}, \bibinfo{author}{Kitajima, H.},
  \bibinfo{author}{Yoshinaga, T.}, \bibinfo{author}{Aihara, K.},
  \bibinfo{author}{Kawakami, H.}, \bibinfo{year}{2006}.
\newblock \bibinfo{title}{Bifurcations in {Morris--Lecar} neuron model}.
\newblock \bibinfo{journal}{Neurocomputing} \bibinfo{volume}{69},
  \bibinfo{pages}{293--316}.
\bibitem[{Ueda et~al.(2021)Ueda, Ohno, Yamamoto, Iwase, Fukuba, Hanamatsu,
  Obama, Ikeda, Ikedo, Yui et~al.}]{ueda2021compressed}
\bibinfo{author}{Ueda, T.}, \bibinfo{author}{Ohno, Y.},
  \bibinfo{author}{Yamamoto, K.}, \bibinfo{author}{Iwase, A.},
  \bibinfo{author}{Fukuba, T.}, \bibinfo{author}{Hanamatsu, S.},
  \bibinfo{author}{Obama, Y.}, \bibinfo{author}{Ikeda, H.},
  \bibinfo{author}{Ikedo, M.}, \bibinfo{author}{Yui, M.}, et~al.,
  \bibinfo{year}{2021}.
\newblock \bibinfo{title}{Compressed sensing and deep learning reconstruction
  for women’s pelvic {MRI} denoising: utility for improving image quality and
  examination time in routine clinical practice}.
\newblock \bibinfo{journal}{European journal of radiology}
  \bibinfo{volume}{134}, \bibinfo{pages}{109430}.
\bibitem[{Wang et~al.(2022)Wang, Li, Xuan, Jiang, Jia, Ji and
  Liu}]{wang2022interpolation}
\bibinfo{author}{Wang, C.}, \bibinfo{author}{Li, X.}, \bibinfo{author}{Xuan,
  K.}, \bibinfo{author}{Jiang, Y.}, \bibinfo{author}{Jia, R.},
  \bibinfo{author}{Ji, J.}, \bibinfo{author}{Liu, J.}, \bibinfo{year}{2022}.
\newblock \bibinfo{title}{Interpolation of soil properties from geostatistical
  priors and {DCT}-based compressed sensing}.
\newblock \bibinfo{journal}{Ecological Indicators} \bibinfo{volume}{140},
  \bibinfo{pages}{109013}.
\bibitem[{Wang et~al.(2023)Wang, Zhang, Chen, He, Li and Wu}]{wang2023brainpy}
\bibinfo{author}{Wang, C.}, \bibinfo{author}{Zhang, T.}, \bibinfo{author}{Chen,
  X.}, \bibinfo{author}{He, S.}, \bibinfo{author}{Li, S.}, \bibinfo{author}{Wu,
  S.}, \bibinfo{year}{2023}.
\newblock \bibinfo{title}{Brainpy, a flexible, integrative, efficient, and
  extensible framework for general-purpose brain dynamics programming}.
\newblock \bibinfo{journal}{Elife} \bibinfo{volume}{12},
  \bibinfo{pages}{e86365}.
\bibitem[{Wang et~al.(2020)Wang, Zhao, Yu, Wang and Bi}]{wang2020structured}
\bibinfo{author}{Wang, L.}, \bibinfo{author}{Zhao, L.}, \bibinfo{author}{Yu,
  L.}, \bibinfo{author}{Wang, J.}, \bibinfo{author}{Bi, G.},
  \bibinfo{year}{2020}.
\newblock \bibinfo{title}{Structured bayesian learning for recovery of
  clustered sparse signal}.
\newblock \bibinfo{journal}{Signal Processing} \bibinfo{volume}{166},
  \bibinfo{pages}{107255}.
\bibitem[{Wang and Buzs{\'a}ki(1996)}]{wang1996gamma}
\bibinfo{author}{Wang, X.J.}, \bibinfo{author}{Buzs{\'a}ki, G.},
  \bibinfo{year}{1996}.
\newblock \bibinfo{title}{Gamma oscillation by synaptic inhibition in a
  hippocampal interneuronal network model}.
\newblock \bibinfo{journal}{Journal of neuroscience} \bibinfo{volume}{16},
  \bibinfo{pages}{6402--6413}.
\bibitem[{Watkins et~al.(2020)Watkins, Kim, Sornborger and
  Kenyon}]{watkins2020using}
\bibinfo{author}{Watkins, Y.}, \bibinfo{author}{Kim, E.},
  \bibinfo{author}{Sornborger, A.}, \bibinfo{author}{Kenyon, G.T.},
  \bibinfo{year}{2020}.
\newblock \bibinfo{title}{Using sinusoidally-modulated noise as a surrogate for
  slow-wave sleep to accomplish stable unsupervised dictionary learning in a
  spike-based sparse coding model}, in: \bibinfo{booktitle}{Proceedings of the
  IEEE/CVF Conference on Computer Vision and Pattern Recognition Workshops},
  pp. \bibinfo{pages}{360--361}.
\bibitem[{Wu et~al.(2023)Wu, Xue, Bao, Yang, Li, Tian, Ren, Li and
  Miao}]{wu2023forward}
\bibinfo{author}{Wu, C.}, \bibinfo{author}{Xue, Y.}, \bibinfo{author}{Bao, H.},
  \bibinfo{author}{Yang, L.}, \bibinfo{author}{Li, J.}, \bibinfo{author}{Tian,
  J.}, \bibinfo{author}{Ren, S.}, \bibinfo{author}{Li, Y.},
  \bibinfo{author}{Miao, X.}, \bibinfo{year}{2023}.
\newblock \bibinfo{title}{Forward stagewise regression with multilevel
  memristor for sparse coding}.
\newblock \bibinfo{journal}{Journal of Semiconductors} \bibinfo{volume}{44},
  \bibinfo{pages}{104101}.
\bibitem[{Xiang et~al.(2021)Xiang, Dong and Yang}]{xiang2021fista}
\bibinfo{author}{Xiang, J.}, \bibinfo{author}{Dong, Y.}, \bibinfo{author}{Yang,
  Y.}, \bibinfo{year}{2021}.
\newblock \bibinfo{title}{{FISTA-Net}: Learning a fast iterative shrinkage
  thresholding network for inverse problems in imaging}.
\newblock \bibinfo{journal}{IEEE Transactions on Medical Imaging}
  \bibinfo{volume}{40}, \bibinfo{pages}{1329--1339}.
\bibitem[{Xu et~al.(2022)Xu, Zhang, Yu, Chen, Xing and Hong}]{xu2022sparse}
\bibinfo{author}{Xu, G.}, \bibinfo{author}{Zhang, B.}, \bibinfo{author}{Yu,
  H.}, \bibinfo{author}{Chen, J.}, \bibinfo{author}{Xing, M.},
  \bibinfo{author}{Hong, W.}, \bibinfo{year}{2022}.
\newblock \bibinfo{title}{Sparse synthetic aperture radar imaging from
  compressed sensing and machine learning: Theories, applications, and trends}.
\newblock \bibinfo{journal}{IEEE Geoscience and Remote Sensing Magazine}
  \bibinfo{volume}{10}, \bibinfo{pages}{32--69}.
\bibitem[{Yi et~al.(2022)Yi, Ran, Tang, Jin, Zhuang, Zhou and
  Lin}]{yi2022improved}
\bibinfo{author}{Yi, C.}, \bibinfo{author}{Ran, L.}, \bibinfo{author}{Tang,
  J.}, \bibinfo{author}{Jin, H.}, \bibinfo{author}{Zhuang, Z.},
  \bibinfo{author}{Zhou, Q.}, \bibinfo{author}{Lin, J.}, \bibinfo{year}{2022}.
\newblock \bibinfo{title}{An improved sparse representation based on local
  orthogonal matching pursuit for bearing compound fault diagnosis}.
\newblock \bibinfo{journal}{IEEE Sensors Journal} \bibinfo{volume}{22},
  \bibinfo{pages}{21911--21923}.
\bibitem[{Zhang et~al.(2023)Zhang, Chen, Xiong and Zhang}]{zhang2023physics}
\bibinfo{author}{Zhang, J.}, \bibinfo{author}{Chen, B.},
  \bibinfo{author}{Xiong, R.}, \bibinfo{author}{Zhang, Y.},
  \bibinfo{year}{2023}.
\newblock \bibinfo{title}{Physics-inspired compressive sensing: Beyond deep
  unrolling}.
\newblock \bibinfo{journal}{IEEE Signal Processing Magazine}
  \bibinfo{volume}{40}, \bibinfo{pages}{58--72}.
\bibitem[{Zhang et~al.(2020)Zhang, Wei, Lu and Pan}]{zhang2020lasso}
\bibinfo{author}{Zhang, L.}, \bibinfo{author}{Wei, X.}, \bibinfo{author}{Lu,
  J.}, \bibinfo{author}{Pan, J.}, \bibinfo{year}{2020}.
\newblock \bibinfo{title}{Lasso regression: from explanation to prediction}.
\newblock \bibinfo{journal}{Advances in Psychological Science}
  \bibinfo{volume}{28}, \bibinfo{pages}{1777}.
\bibitem[{Zhang et~al.(2022)Zhang, Yu, Zheng and Eldar}]{zhang2022spiking}
\bibinfo{author}{Zhang, X.}, \bibinfo{author}{Yu, L.}, \bibinfo{author}{Zheng,
  G.}, \bibinfo{author}{Eldar, Y.C.}, \bibinfo{year}{2022}.
\newblock \bibinfo{title}{Spiking sparse recovery with non-convex penalties}.
\newblock \bibinfo{journal}{IEEE transactions on signal processing}
  \bibinfo{volume}{70}, \bibinfo{pages}{6272--6285}.
\bibitem[{Zhu and Rozell(2013)}]{zhu2013visual}
\bibinfo{author}{Zhu, M.}, \bibinfo{author}{Rozell, C.J.},
  \bibinfo{year}{2013}.
\newblock \bibinfo{title}{Visual nonclassical receptive field effects emerge
  from sparse coding in a dynamical system}.
\newblock \bibinfo{journal}{PLoS computational biology} \bibinfo{volume}{9},
  \bibinfo{pages}{e1003191}.
\bibitem[{Zins(2023)}]{zins2023neuromorphic}
\bibinfo{author}{Zins, N.}, \bibinfo{year}{2023}.
\newblock \bibinfo{title}{Neuromorphic computing applications in robotics}.
\newblock Ph.D. thesis. Michigan Technological University.
\bibitem[{Zou and Hastie(2005)}]{zou2005regularization}
\bibinfo{author}{Zou, H.}, \bibinfo{author}{Hastie, T.}, \bibinfo{year}{2005}.
\newblock \bibinfo{title}{Regularization and variable selection via the elastic
  net}.
\newblock \bibinfo{journal}{Journal of the royal statistical society: series B
  (statistical methodology)} \bibinfo{volume}{67}, \bibinfo{pages}{301--320}.

\end{thebibliography}
\end{document}